\documentclass[12pt]{amsart}
\usepackage{amsmath,amssymb,amsthm}
\usepackage{mathrsfs}

\usepackage[margin=30mm]{geometry}

\RequirePackage{color}
\RequirePackage[colorlinks,urlcolor=my-blue,linkcolor=my-red,citecolor=my-green]{hyperref}
\definecolor{my-blue}{rgb}{0.0,0.0,0.6}
\definecolor{my-red}{rgb}{0.5,0.0,0.0}
\definecolor{my-green}{rgb}{0.0,0.5,0.0}

\newtheorem{theorem}{\sc Theorem}[section]
\newtheorem{lemma}[theorem]{\sc Lemma}
\newtheorem{proposition}[theorem]{\sc Proposition}

\numberwithin{equation}{section}
\theoremstyle{remark}
\newtheorem{remark}[theorem]{Remark}

\newcommand{\be}{\begin{equation}}
\newcommand{\ee}{\end{equation}}
\newcommand{\nn}{\nonumber}
\providecommand{\abs}[1]{\vert#1\vert}

\def\bE{\mathbb{E}}
\def\bN{\mathbb{N}}
\def\bP{\mathbb{P}}

\def\bR{\mathbb{R}}
\def\bZ{\mathbb{Z}}

\def\w{\omega}      
\def\om{\omega}

\def\e{\varepsilon}
\def\ind{\mathbf{1}}

\def\m1{\mathbf{1}}

\newcommand{\she}{\mathcal{Z}}
\newcommand{\wh}{\mathscr{W}}
\newcommand{\cB}{\mathcal{B}}
\newcommand{\vf}{\varphi}
\newcommand{\free}{\mathbf{F}}

 \def\Vvv{{\rm\mathbb{V}ar}}

 \def\wt{\widetilde}  

\def\E{\bE}
\def\P{\bP} 
 
\def\digamf{\Psi_0} 
\def\trigamf{\Psi_1} 
\newcommand{\eqd}{\stackrel{d}{=}}
\allowdisplaybreaks[1]
\def\eps{\varepsilon}
\def\para{\theta} 

  \def\bbb{b}      

 \newcommand{\bea}{\begin{eqnarray}}
\newcommand{\eea}{\end{eqnarray}}
\newcommand{\ben}{\begin{eqnarray*}}
\newcommand{\een}{\end{eqnarray*}}
 \def\tZ{\tilde Z}

\newcommand{\evmax}{\lambda^n_{\textup{max}}}

\def\betaa{\beta_0}  

\begin{document}

\title[Fluctuation exponents for directed polymers]{Fluctuation exponents for directed polymers in the intermediate 
disorder regime}

\author{Gregorio R. Moreno Flores}
\address{Gregorio R. Moreno Flores\\ Pontificia Universidad Cat\'olica de Chile \\
Departamento de Matem\'aticas \\ Vicu\~na Mackenna 4860 \\ Macul - Santiago \\ Chile.}
\email{grmoreno@mat.puc.cl}
\thanks{G. Moreno was partially supported by a Fondecyt grant 1130280.}
\author[T.~Sepp\"al\"ainen]{Timo Sepp\"al\"ainen}
\address{Timo Sepp\"al\"ainen\\ University of Wisconsin-Madison\\  Mathematics Department\\ Van Vleck Hall\\ 480 Lincoln Dr.\\   Madison WI 53706-1388\\ USA.}
\email{seppalai@math.wisc.edu}
\urladdr{http://www.math.wisc.edu/~seppalai}
\thanks{T.\ Sepp\"al\"ainen was partially supported by  National Science Foundation grants DMS-1003651 and DMS-1306777 and by the Wisconsin Alumni Research Foundation.} 
\author{Benedek Valk\'o}
\address{Benedek Valk\'o\\ University of Wisconsin-Madison\\  Mathematics Department\\ Van Vleck Hall\\ 480 Lincoln Dr.\\   Madison WI 53706-1388\\ USA.}
\email{valko@math.wisc.edu}
\thanks{B. Valk\'o was partially supported by  National Science Foundation CAREER award DMS-1053280.}

\begin{abstract}
We compute the fluctuation exponents for a solvable model of one-dimensional directed polymers in random environment in the intermediate regime. This regime corresponds to taking the inverse temperature to zero with the size of the system. The exponents satisfy the KPZ scaling relation and coincide with physical predictions.
In the critical case, we recover the fluctuation exponents of the Cole-Hopf solution of the KPZ equation in equilibrium and close to equilibrium.  
\end{abstract}

\maketitle
%
%
%
%
%
%
%
%
%
%
%
%
%
%
%
%

 \section{Main results}
 
\subsection{Introduction}

The {\sl directed polymer in a random environment} is a 
statistical physics model that    assigns  Boltzmann-Gibbs weights 
to   random walk
paths as a function of  the environment
encountered by the walk.  It was  originally introduced in  \cite{HH} as a model of an interface in two dimensions.  Here is the standard lattice formulation in $d+1$ dimensions ($d$ space dimensions, one time dimension).  

The environment is a collection of i.i.d.\ random weights $\{\om(i,x):\,
i\in\mathbb{N},\,x\in \mathbb{Z}^d\}$ with probability distribution $\P$. 
Let $P$ be the law of   simple symmetric random walk $(S_t)_{t\in\mathbb{Z}_+}$ on $\mathbb{Z}^d$ with $S_0=0$.    Denote   expectation under  $P$ and $\P$ by $E$ and $\E$,   respectively.   The {\sl quenched  partition function} of the directed polymer   
 in environment $\om$ and  at inverse temperature
$\beta>0$ is   
\begin{eqnarray}\label{pf-dpre}
 Z_{N,x}(\beta) = E\bigl[ e^{\beta \sum^N_{i=1}\om(i,S_i)},\, S_N=x\bigr],
\end{eqnarray} 
where $E[X,A]$ is the expectation of $X$ restricted to  the event $A$. 
 This is the {\sl point-to-point}  partition function because the endpoint $S_N$ of the walk is
 constrained to be $x$.  The version that allows $S_N$ to fluctuate freely is the
 {\sl point-to-line} partition function.  
 In the point-to-point setting the {\sl quenched polymer measure}   
   on paths ending at $x$ is 
\be \begin{aligned} 
 &Q^{\beta}_{N,x}\left(S_1=x_1,\dotsc,\, S_N=x_N\right) \\
&\qquad  =  \frac{1}{Z_{N,x}(\beta)}e^{\beta \sum^N_{i=1}\om(i,x_i)} P\left[ S_1=x_1,\dotsc, S_N=x_N\right] \cdot \ind\{x_N=x\} . \end{aligned} 
\label{Q1}\ee
These quenched quantities are functions of the environment $\w$ and thereby random.  
The averaged distribution of the path is $P^{\beta}_{N,x}(\cdot)=\E Q^{\beta}_{N,x}(\cdot)$. 
We refer the reader to  reviews  \cite{CSY, denH, KS}  for  a deeper discussion 
 of the subject. 

  We restrict the  discussion to the 1+1 dimensional case.  
Basic objects of study are the fluctuations of the free energy 
$\log Z_{N,Nx}(\beta)$ and the path $(S_t)_{0\le t\le N}$. 
On the crudest 
level the orders of magnitude of these  fluctuations are 
described by two exponents $\chi$ and $\zeta$:
\begin{itemize}
\item fluctuations of $\log Z_{N, Nx}(\beta)$ under $\P$ have order of magnitude
$N^\chi$
\item fluctuations of the path $S_t$ under $P_{N, Nx}^\beta$ have order of magnitude
$N^\zeta$
\end{itemize}
In the 1+1 dimensional case these exponents are expected to take
the values $\chi=1/3$ and $\zeta=2/3$  independently
of $\beta$, provided the i.i.d.\ weights $\om(i,x)$ satisfy a moment bound.    Furthermore,  there are specific predictions for the limit distributions 
of the scaled quantities:  for example, the GUE Tracy-Widom distribution  for 
$\log Z_{N, Nx}(\beta)$.  These properties are features  of the {\sl Kardar-Parisi-Zhang (KPZ)
universality class} to which these models are expected to belong.  
See \cite{Ivan, S-SI} for  recent surveys.   The KPZ regime should be contrasted
with the {\sl diffusive regime} where $\chi=0$, $\zeta=1/2$, and the
path satisfies a central limit theorem.  Diffusive behavior is known to happen
for $d\ge 3$ and small enough $\beta$ \cite{CY}. 

There are three exactly solvable 1+1 dimensional  models for which KPZ predictions 
have been partially proved:  
 \begin{enumerate}
\item[(a)] 
 the semidiscrete polymer in a Brownian
environment   \cite{OY}    
\item[(b)]  the log-gamma polymer \cite{S}   
\item[(c)]   the continuum directed random polymer,
in other words, the solution of the Kardar-Parisi-Zhang (KPZ) equation 
\cite{AKQ2, ACQ, KPZ}  
\end{enumerate} 
In  recent years a number of 
results have appeared, 
first for  exponents and then for distributional properties.   
This is not  a  place for 
a thorough  review, but let us cite some of the relevant papers:
\cite{ACQ, BC, BCF, BCR, BQS, cosz, O, S, SV}.   To do justice to history,  
we mention also  that KPZ results appeared earlier for zero-temperature polymers
(the $\beta\to\infty$ limit of \eqref{pf-dpre}--\eqref{Q1}, known
as last-passage percolation), beginning with the seminal papers
\cite{BDJ, J}. 

Getting closer to the topic of the  present paper, physics paper
 \cite{AKQ0} introduced the study of the {\sl intermediate disorder regime}
 in model \eqref{pf-dpre}--\eqref{Q1}. This means that $\beta$ 
 is scaled to zero as $N\to\infty$ by taking $\beta=\beta_0N^{-\alpha}$. 
 The window of interest is $0\le \alpha\le 1/4$.  At $\alpha=0$ one sees the
 KPZ behavior with exponents $\chi=1/3$ and $\zeta=2/3$.  At $\alpha=1/4$ one has the critical case where exponents are   diffusive  
 ($\chi=0$ and $\zeta=1/2$) but fluctuations are different  \cite{AKQ}. 
  When
$\alpha>1/4$ the disorder is   irrelevant and the polymer behaves like a
simple random walk   \cite{AKQ2}.  

Article  \cite{AKQ0} conjectured the exponents for the entire range:
\be   \chi(\alpha)=\tfrac13 (1-4\alpha)
\quad\text{and}\quad 
 \zeta(\alpha)=\tfrac23 (1-\alpha)\quad \text{for $0\le \alpha\le 1/4$. }  \label{exp1}\ee
  In this paper we derive these  intermediate disorder  exponents  
for the semidiscrete polymer in the Brownian environment  (introduced in \cite{OY}, hence  also called the O'Connell-Yor model). Along the way we offer some improvements to the 
earlier work \cite{SV} which treated the $\alpha=0$ case.    This model has 
two versions:  a {\sl stationary version} with  particular
 boundary conditions that render the process of $\log Z$ increments
 shift-invariant, and the 
{\sl point-to-point version}  without boundary conditions represented  by 
\eqref{pf-dpre}--\eqref{Q1}  above. In general we have  better results for the stationary version.  
  In case the reader is encountering   polymer
models with boundaries for the first time but can appreciate an analogy
 with the 
totally asymmetric simple exclusion process (TASEP), then 
the stationary polymer corresponds to stationary TASEP with Bernoulli
occupations, while the point-to-point version of the polymer is the analogue
of TASEP with  step initial condition.  

We list below the precise  contributions of our  paper: 
 \begin{enumerate}
\item[(i)]   For the free energy we derive the exponent 
  $\chi(\alpha)=\tfrac13 (1-4\alpha)$
for the entire range $0\le \alpha\le 1/4$ for the stationary version  and 
for $0\le \alpha<1/4$ for the point-to-point version.  
 For the fixed temperature case ($\alpha=0$)
the lower bound $\chi\ge 1/3$ for the point-to-point version  was not covered in \cite{SV}, but is
done here.   
\item[(ii)] We have  the path exponent $\zeta(\alpha)=\tfrac23 (1-\alpha)$
for the stationary version, and the upper bound 
$\zeta(\alpha)\le\tfrac23 (1-\alpha)$ for the point-to-point version. 
\item[(iii)]  Our results refine the prediction \eqref{exp1} in the following way. 
 We introduce a macroscopic time parameter $\tau>0$ and conclude 
 that  the fluctuations of $\log Z_{\tau N,\tau Nx}(\beta_0N^{-\alpha})$ are 
 of magnitude $\tau^{1/3}N^{\chi(\alpha)}$ while the path fluctuations
are of magnitude $\tau^{2/3}N^{\zeta(\alpha)}$.  In other words,  in the
macroscopic variables we see again the exponents $\tfrac13$ and $\tfrac23$. 
\item[(iv)]   In the fixed temperature case ($\alpha=0$) the lower bound $\chi\ge 1/3$ 
was already proved in \cite{SV} for the stationary version.  Here we give a 
 considerably simpler proof of the lower bound, including the case $\alpha=0$.  
\item[(v)]   In the critical case $\alpha=1/4$ we can connect with
the KPZ equation.  The macroscopic variable $\tau$ becomes the time parameter
of the stochastic heat equation (SHE), and we obtain again the exponent of
the  stationary Hopf-Cole solution of the  KPZ equation,  first proved in \cite{BQS}:
$\Vvv [\log \she(\tau,0)] \asymp \tau^{\frac23}$ where $\she$ is the
solution of SHE. Moreover, we prove similar bounds for solutions where the initial condition is a bounded perturbation of the stationary initial condition. 
\end{enumerate}  

Some further comments about the  state of the field and the
place of this work are in order.    Presently one can identify the 
following  three
approaches to  fluctuations of   polymer models and of models in
the KPZ class more broadly.  
\begin{enumerate}
\item[(a)]  The \emph{resolvent method}.  This is a fairly robust method used to establish 
superdiffusivity.  It is quite general, for it can often be applied as long as a
model has a tractable invariant distribution \cite{B1, LQSY, QV1, QV2, QV3,  Yau}. 
A drawback of the method is that often it cannot determine the exact exponents but
  provides only  bounds on them.  However, here are two exceptions.  
   In \cite{Yau} the scaling exponent of a 2d TASEP model is identified exactly. In \cite{QV1, QV2} the method is used to give a comparison between the solvable 1d TASEP and more general 1d exclusion models to show that the scaling exponents are the same.  

  \item[(b)]   The \emph{coupling method}, represented by the present work
  and references \cite{BCS, BKS, BQS, BS, CG,  S, SV}.  
 This approach is able to identify exact exponents, but so far has
 depended  on the
 presence of special structures such as a Burke-type property. 
 \item[(c)]   \emph{Exact solvability methods}.  When it can be applied, 
this approach leads to the sharpest 
results, namely   Tracy-Widom limit distributions.   But it is the
most specialized and technically very heavy.   This approach became
available for the semidiscrete polymers  after 
determinantal expressions where found for the distribution of $\log Z$
\cite{BC, BCF, O}. For the related log-gamma polymers, see \cite{BCR, cosz}. For the ASEP the first scaling limits were proved using Fredholm determinant formulas based on the work of \cite{trac-wido-08jsp}. The recent work of \cite{BoCS} uses  certain duality relations to get scaling limits for the same model.  Their method can be thought of as a rigorous version of the so-called `replica trick'. 
 \end{enumerate} 

The free energy exponent $\chi=1/3$ in the fixed temperature case  ($\alpha=0$) 
is   also  a consequence of   the distributional limits for $\log Z$  in  
 \cite{BC, BCF}.    Presently these results cover
the point-to-point case of the semidiscrete polymer for the entire fixed temperature range $0<\beta<\infty$.  
It is expected that these methods should work also in the intermediate disorder
regime (personal communication from the authors). However, these works
do not yet give anything on the stationary versions of the models, or on the
path fluctuations in either the point-to-point or stationary version.    And even in the cases covered by the
distributional limits,  our results give sharper bounds on moments and 
deviations that  weak limits alone cannot provide.   
 
 The open problem that remains  in the coupling approach used here is
 the lower bound for the path in the point-to-point case.

\smallskip 
 
 One more  expected universal feature of polymer exponents worth
 highlighting here  is the scaling
relation $\chi=2\zeta-1$.  This is expected to hold very generally across models and dimensions. 
The exponents we derive satisfy this identity.  There is important recent
work on this identity that goes beyond exactly solvable models:
first   \cite{Ch}, and then \cite{AD} with a simplified proof, 
derived this relation for first passage percolation under   strong
assumptions on the existence of the exponents. 
These results are extended to positive temperature  
 directed polymers   in   \cite{AD2}.
 
 Finally, we point out that the coupling method applied to directed polymers first appeared in the work \cite{S} in the  context of discrete polymers in a log-gamma environment. Most of the results of \cite{SV} have discrete analogues in \cite{S}. The intermediate regime can also be investigated for the polymers in log-gamma environment. Although this model is formulated for $\beta=1$, the parameters of the environment can be tuned to emulate the situation $\beta\to0$. We have obtained proofs for the fluctuation exponents of the log-gamma model in the intermediate scaling regime. 
 The methods are very similar to the ones used here for the semidiscrete polymer model, but involve considerably heavier asymptotics   so we decided not to include them in the present paper.  
 

\vspace{2ex}
\noindent  \textsc{Organization of the paper.} We introduce the   directed polymer in a Brownian environment in its point-to-point and stationary versions and state our main theorems in Sections  \ref{sec-ptp-dpbe} and \ref{sec-stat-dpbe}.   Their proofs are  in Section \ref{pf-dpbe}.
In Section \ref{sec-KPZ} we state our results for the KPZ equation. The proofs are given in Section \ref{pf-KPZ}. Some basic estimates on polygamma functions are provided in Section \ref{gtb}.

\vspace{2ex}

\noindent  \textsc{Notation and conventions.} 
$\bN=\{1,2,3,\dotsc\}$ and $\bZ_+=\{0,1,2,\dotsc\}$. 
For $\theta>0$,  
the usual gamma function is $\Gamma(\theta)=\int_0^\infty s^{\theta-1}e^{-s}\,ds$ 
 and the Gamma($\theta, r$) distribution has density 
$f(x)=r^\theta\Gamma(\theta)^{-1} x^{\theta-1}e^{-rx}$ for $0<x<\infty$. If $r=1$ then we drop it from the notation, i.e.~Gamma$(\theta)$ is the same as Gamma$(\theta,1)$.
We use the notations    $\Psi_0=\Gamma'/\Gamma$ and $\Psi_1=\Psi_0'$  are the digamma and trigamma functions.  $\Psi_1^{-1}$
  is the inverse function of $\Psi_1$. See Section \ref{gtb} for more on polygamma functions and their relations to the Gamma$(\theta, r)$ distribution. 

The environment distribution $\P$ has expectation symbol $\E$. Generically expectation under a
 probability measure $Q$ is denoted by $E^Q$.    
  To simplify notation we drop integer parts.  A real value $s$ in a position
  that takes an integer should be interpreted as the  integer part $\lfloor s\rfloor$.  
 
 \vspace{2ex}

\noindent  \textsc{Aknowledgements.}   The authors 
thank Michael Damron for the decomposition  idea
in  the proof of Theorem \ref{thm:freeZa}.

\subsection{The semi-discrete polymer in a Brownian environment} 
We begin with the results for   the semi-discrete polymer in a Brownian environment. This is a semi-discrete version of the generic polymer model described in (\ref{pf-dpre}). As already mentioned, the model has two versions: a point-to-point and a stationary version.

\subsubsection{Point-to-point semi-discrete polymer}\label{sec-ptp-dpbe}
The environment consists of a family of independent one-dimensional
standard Brownian motions $\{B_i(\cdot): i\geq 1\}$. These are two-sided  Brownian motions with   
  $B_i(0)=0$. 
  Polymer paths are nondecreasing 
  c\`adl\`ag paths  $x: [0,t]\to \bN$ 
with nearest-neighbor jumps,   $x(0)=1$, and $x(t)=n$.    
A path can be coded in terms of its jump times 
  $0=s_0<s_1<\cdots<s_{n-1}<s_n=t$.  
At level $k$
 the path  collects the increment  
$B_k(s_{k-1},s_k)=B_k(s_k)-B_k(s_{k-1})$. 
The partition function in a fixed  Brownian environment at
inverse temperature $\beta>0$ is, for $(n,t)\in\bN\times[0,\infty)$, 
\be \label{pf-neil}
 \begin{aligned} 
Z_{n,t}(\beta)&=   
\int\limits_{0<s_{1}<\dotsm<s_{n-1}<t}  \exp\bigl[ \beta\bigl(B_1(0,s_1) +B_{2}(s_1,s_{2})
+\dotsm + B_n(s_{n-1},t)\bigr)\bigr] \,ds_{1,n-1} . 
\end{aligned}
\ee
In the integral $ds_{1,n-1}$ is short for $ds_1\dotsm ds_{n-1}$.  
The limiting  free energy density  was computed  for a fixed $\beta$ in \cite{MO}:  
\be\begin{aligned}\label{free-MOC}
\free(\beta) = \lim_{n\to +\infty} \frac{1}{n} \log Z_{n,n}(\beta) 
&= \inf_{t>0}\left\{ t  \beta^2 -\Psi_0(t)
\right\}-2\log \beta\\&=\Psi_1^{-1}(\beta^2)  \beta^2 -\Psi_0(\Psi_1^{-1}(\beta^2) )
-2\log \beta \quad \text{ for $\beta>0$.}   
\end{aligned}\ee


We consider this model 
in the intermediate disorder regime where
$\beta=\betaa n^{-\alpha}$ for   fixed $\betaa\in(0,\infty)$ and  $\alpha\in[0,1/4]$. 
If $0<\alpha\le 1/4$,  $\log Z_{n,n}(\beta)$    concentrates asymptotically around the value
$n \free(\beta_0 n^{-\alpha})
= n+O(n^{1-2\alpha})$.  
(See \eqref{didam1} and \eqref{Psi-3} for the asymptotics of the functions $\Psi_0$ and $\Psi_1$.) 

Our first result  identifies the  free energy fluctuation exponent $\chi=\frac13(1-4\alpha)$ for the point-to-point semi-discrete polymer in the intermediate disorder
regime.   In  the fixed temperature case ($\alpha=0$) the upper bound was proved
in \cite{SV} but  a lower bound proof with coupling methods  is new even in this case.   (To clarify,  the correct exponent in the $\alpha=0$ case  has  of course  been identified in the weak convergence results  \cite{BC, BCF} with exact solvability methods.)    Note that we see
the intermediate regime exponent on the scaling parameter $n$,  but   for the 
macroscopic variable $\tau$ we see the  exponent $\tfrac13$ corresponding to the KPZ scaling.  

\begin{theorem}\label{thm:freeZ} Fix $\alpha\in[0,1/4)$ and $0<\beta_0<\infty$. 
Let $\beta=\betaa n^{-\alpha}$.  
There exist finite positive constants $C, n_0, b_0, \tau_0$ that depend on 
$(\alpha,\beta_0)$ 
such that the following bounds hold.  
 For   $\tau \geq \tau_0,\, n\geq n_0$ and $b\geq b_0$, 
\be \label{free0}
\P\bigl\{\abs{\log Z_{\tau n,\tau n}(\beta)   -\tau n \free(\beta) }
\ge b \, \tau^{\frac13}n^{\frac13(1-4\alpha)}\bigr\}\le C b^{-3/2}   
\ee 
and 
 \be\label{free1}
 C^{-1} \tau^{\frac13} n^{\frac13(1-4\alpha)} \;\leq \; 
 \E\abs{\log Z_{\tau n, \tau n}(\beta) 
-\tau n \free(\beta) } \; \leq \; C \tau^{\frac13} n^{\frac13(1-4\alpha)}. 
\ee
 \end{theorem}
 
 \smallskip
 
We turn to the  fluctuations of the polymer path. The 
quenched polymer measure
$Q_{n,t, \beta}$ on paths is defined, in terms of the expectation
of  a bounded Borel  function  $f:\bR^{n-1}\to\bR$,  by 
\begin{align*}
E^{Q_{n,t,\beta}}f(\sigma_1,\dotsc,\sigma_{n-1})&=\frac{1}{Z_{n,t}(\beta)}   
\int\limits_{0<s_{1}<\dotsm<s_{n-1}<t} f(s_1,\dotsc,s_{n-1}) \\[2pt] 
&\qquad\qquad \times \  \exp\bigl[ \beta\bigl(B_1(0,s_1)
+\dotsm + B_n(s_{n-1},t)\bigr)\bigr] \,ds_{1,n-1} . 
\end{align*}
  The jump
times as functions of the path are denoted  by $\sigma_i$. 
Averaged (or \emph{annealed}) probability and expectation are   denoted by  
$P_{n,t, \beta}(\cdot)=\bE Q_{n,t, \beta}(\cdot)$  and 
 $E_{n,t,\beta}(\cdot)=\bE E^{Q_{n,t,\beta}}(\cdot)$. 
 
In the point-to-point setting the path exponent $\zeta$ describes the order
of magnitude of the  deviations 
of the path  from the diagonal.  A path close to the diagonal in the
rectangle $\{1,\dotsc,n\}\times[0,t]$  would have $\sigma_i\approx it/n$.  
 The next theorem shows that the  path  
exponent $\zeta$ is bounded above by its conjectured value
$\frac23(1-\alpha)$.    

\begin{theorem}\label{thm:path-ptp} 
Fix $\alpha\in[0,1/4)$ and $0<\beta_0<\infty$. 
Let $\beta=\betaa n^{-\alpha}$.  
There exist finite positive constants $C, n_0, b_0, \tau_0$ that depend on 
$(\alpha,\beta_0)$ 
such that the following bound holds.  For all $0<\gamma<1$,  $\tau \geq \tau_0$, $b\ge b_0$,
 and $ n\geq n_0$, 
\begin{eqnarray}
 P_{n,t,\beta}\left\{ \abs{\sigma_{\gamma \tau n}  - \gamma \tau n} \geq b \,\tau^{\frac23} n^{\frac23
(1-\alpha)} \right\}   
 \leq C b^{-3}. 
\end{eqnarray}
 \end{theorem}


\smallskip

\subsubsection{Stationary  semi-discrete polymer}\label{sec-stat-dpbe}
The proofs of the above theorems rely on comparison with a stationary version of the model.
Enlarge the environment by adding another 
Brownian motion $B$ independent of $\{B_i\}_{i\geq 1}$. Introduce a parameter
$\para\in(0,\infty)$ and restrict to $\beta=1$ for a moment. The stationary partition function is,  for $n\in\bN$ and $t\in\bR$, 
\be\begin{aligned}   Z^\para_{n,t} &=\hskip-20pt 
\int\limits_{-\infty<s_{0}<s_1<\dotsm<s_{n-1}<t}  \hskip-20pt \exp\bigl[ -B(s_0)+\para s_0  +  B_1(s_0,s_1) 
++\dotsm + B_n(s_{n-1},t)\bigr] \,ds_{0,n-1}.
\end{aligned}\label{Zdef2}\ee
This model has a useful stationary structure
described by \cite{OY}.  Let $Y_0(t)=B(t)$ and, for $k\geq
1$, define inductively 
\begin{eqnarray}\label{sc-r}
 r_k(t)&=&\log \int^t_{-\infty} e^{Y_{k-1}(s,t)-\theta(t-s)+B_k(s,t)}ds\\
 \label{sc-Y}
 Y_k(t)&=&Y_{k-1}(t)+r_k(0)-r_k(t).
\end{eqnarray}
Induction  shows that
\begin{eqnarray}\label{Zr-ident}
 Z^{\theta}_{n,t}\, e^{B(t)-\theta t} &=& \exp\Bigl(\; \sum^n_{k=1} r_k(t)\Bigr) .
\end{eqnarray}
  For each fixed   $t\ge 0$,  $\{r_k(t)\}_{k\ge 1}$ are
i.i.d.\ and $e^{-r_k(t)}$ has  Gamma$(\theta)$ distribution \cite{OY}. Thus the law of $Z^{\theta}_{n,t}\, e^{B(t)-\theta t}$ is  independent of $t$.   This stationarity is part of a broader Burke-type property (see  
  \cite[Section 3.1]{SV} for more details).


Extend   definition (\ref{pf-neil})  
to $1\le k\le n \in \bN$ and $s<t\in\bR$ by 
\be\begin{aligned} 
Z_{(k,n),(s,t)}(\beta)&=   
\int\limits_{s<s_{k}<\dotsm<s_{n-1}<t}  \exp\bigl[ \beta\bigl(B_k(s,s_k) +B_{k+1}(s_k,s_{k+1})
+\dotsm + B_n(s_{n-1},t)\bigr)\bigr] \,ds_{k,n-1}, 
\end{aligned}\label{zetadef1}\ee
and abbreviate the $\beta=1$ case as  $Z_{(k,n),(s,t)}=Z_{(k,n),(s,t)}(1)$.  The stationary
partition function can be recovered by integrating these point-to-point partition functions against
the boundary Brownian motion:
\[ 
 Z^\para_{n,t} =
\int^t_{-\infty}ds_0 \,   e^{ -B(s_0)+\para s_0}
 \, Z_{(1,n),(s_0,t)}.
\]
We include 
the inverse temperature in   the stationary
partition function by defining 
\be\begin{aligned}   Z^{\para,\beta}_{n,t} &=
\int\limits_{-\infty<s_{0}<s_1<\dotsm<s_{n-1}<t}  \exp\bigl[ -\beta B(s_0)+\beta\para s_0   + \beta\bigl( B_1(s_0,s_1)  +\dotsm + B_n(s_{n-1},t)\bigr) \bigr] \,ds_{0,n-1}.
\end{aligned}\label{Zdef2.1}\ee

The following theorem identifies the fluctuation exponent $\chi$ for the stationary model.  
A key difference between the point-to-point and stationary versions is that KPZ fluctuations appear in the 
stationary version only in a particular characteristic direction $(n,t)$ 
 determined by the parameters.   In other directions the diffusive fluctuations
of the boundaries dominate (see \cite{S}, Corollary 2.2, in the context of discrete polymers in a log-gamma environment).  
Once we choose 
$\beta=\beta_0n^{-\alpha}$,  to make the diagonal a characteristic direction
 we are forced
to pick $\theta  = \beta \Psi^{-1}_1(\beta^2)\sim \beta^{-1}$.  To simplify notation we suppress   the  $n$-dependence
of the parameters  $\beta$ and $\theta$.   
\begin{theorem}\label{thm-unscaled} Let  $\alpha\in[0,1/4]$, $\beta=  \beta_0n^{-\alpha}$,
and $\theta = \beta \Psi^{-1}_1(\beta^2)$. Then there exist positive constants $C_1,\,
C_2, \tau_0$ depending only on  $\alpha$ and $\beta_0$ such that
  \be 
C_1 \tau^{\frac23} n^{\frac23(1-4\alpha)} \le \Vvv (\log Z_{\tau n,\tau n}^{\para, \beta})\le  C_2
\tau^{\frac23} n^{\frac23(1-4\alpha)} 
\ee
for all $\tau\geq \tau_0$ and $n\ge 1$.
\end{theorem}
The stationary  quenched polymer
measure $Q^{\para,\beta}_{n,t}$ lives on nondecreasing c\'adl\'ag
paths $x:(-\infty,t]\to\{0,1,\dotsc,n\}$ with boundary conditions   $x(-\infty)=0$, $x(t)=n$. We
represent paths again in terms of   jump times $-\infty<\sigma_0<\sigma_1<\dotsm<\sigma_{n-1}\le t $ {where}
$x(\sigma_i-) = i< i+1=x(\sigma_i)$. 
The path  measure is defined   by 
\be\begin{aligned} 
&E^{Q^{\para,\beta}_{n,t}}f(\sigma_0,\sigma_1,\dotsc,\sigma_{n-1})  
=\frac1{Z_{n,t}^{\para,\beta} } 
\int\limits_{-\infty<s_{0}<\dotsm<s_{n-1}<t}  f(s_0, s_1,\dotsc, s_{n-1}) 
 \\[7pt]
 &\hskip50pt  \times  \exp\bigl[ -\beta B(s_0)+\beta \para s_0 +  \beta \left(B_1(s_0,s_1) 
+\dotsm + B_n(s_{n-1},t)\right)\bigr] \,ds_{0,n-1}.
\end{aligned}\label{Qdef2}
\ee
Averaged probability and expectation are   denoted by  
$P^{\para,\beta}_{n,t}(\cdot)=\bE Q^{\para,\beta}_{n,t}(\cdot)$  and 
 $E^{\para,\beta}_{n,t}(\cdot)=\bE E^{Q^{\para,\beta}_{n,t}}(\cdot)$. When $\beta=1$, we simply remove it from the notation.

In the stationary case we can identify the exact   path  exponent  
$\zeta=\frac23(1-\alpha)$. 
\begin{theorem}\label{thm:statpath} Let $\alpha\in[0,\frac14]$, $\beta= \beta_0 n^{-\alpha}$ and 
$\theta =  \beta \Psi^{-1}_1(\beta^2)$.   The following bounds hold
for all $\gamma\in (0,1)$ and  $0<\tau<\infty$.

{\rm (Upper bound on the tail.)}  
There is a constant $C=C(\alpha,\beta_0)>0$  such that,
if   $\tau n\ge 1$ and $b\ge 1$
  \begin{align}
P^{\para,\beta}_{\tau n,\tau n}\left\{ |\sigma_{\gamma \tau n}-\gamma \tau n|>b
\tau^{\frac23}n^{\frac23(1-\alpha)}\right\} \leq
C b^{-3}.\label{statpathUB}
\end{align}

{\rm (Bounds on the first absolute moment.)}  

\begin{align}
\label{statpath3} C_1^{-1} \tau^{\frac23}n^{\frac23(1-\alpha)}\le  E^{\para,\beta}_{\tau n,\tau n} |\sigma_{\gamma \tau n}-\gamma \tau n|\le C_1 \tau^{\frac23}n^{\frac23(1-\alpha)}
\end{align}

\end{theorem}


\smallskip

\subsection{The KPZ equation close to equilibrium}\label{sec-KPZ}
The Kardar-Parisi-Zhang (KPZ) equation was introduced in \cite{KPZ} as a model of a randomly growing interface in $1+1$ dimension: if we let $h(t,x)$ denote the height of the interface at site $x\in \bR$ and time $t\geq 0$, then the evolution of the interface can be represented by the stochastic partial differential equation
\begin{eqnarray}\label{KPZ}
 \partial_t h &=& \tfrac12 \Delta h + \tfrac12 \left( \nabla h\right)^2 + \wh,
\end{eqnarray}
where $\wh$ is a space-time white noise. We will be interested in initial conditions of the form $\cB+\vf$ where $\cB$ is a double-sided one-dimensional Brownian motion and $\vf$ is a bounded function. It is well known that the meaning of a solution to equation (\ref{KPZ}) is a delicate matter. We will always consider the so called Cole-Hopf solution: we take $h=-\log \she$, where $\she$ solves the stochastic heat equation
\begin{eqnarray}\label{shevf1}
 \partial_t \she^{\vf} &=& \tfrac12 \Delta \she^{\vf} + \wh \she^{\vf}\\
 \label{shevf2}
 \she^{\vf}(0,x)&=&e^{\vf(x)+\cB(x)}.  
\end{eqnarray}
The relation between (\ref{KPZ}) and (\ref{shevf1})-(\ref{shevf2}) can be seen by a formal application of It\^o's formula. For a more detailed overview of the KPZ equation, we refer the reader to the review \cite{Ivan} and its references. See also \cite{H} for   recent rigorous work on solving  (\ref{KPZ}).   

It is expected that, for a wide family of initial conditions, the fluctuations of $\log \she(t,x)$ are of order $t^{1/3}$. This was first proved in \cite{BQS} in the stationary case with $\vf=0$. The proof was based on  the convergence of the rescaled height function of the weakly asymmetric exclusion process to the Cole-Hopf solution of the KPZ equation \cite{BG},  together with non-asymptotic fluctuation bounds on the current of the asymmetric simple exclusion process \cite{BS}. It is not clear that this approach can be  extended to the case of non-zero $\vf$. 

When the initial condition is $\she(0,x)=\delta_0(x)$, the asymptotic distribution of the fluctuations of $\log \she$ is identified in \cite{ACQ} as the Tracy-Widom distribution. The proof is based on heavy asymptotic analysis of exact formulas for the weakly asymmetric simple  exclusion process.

We will extend the result of \cite{BQS} to the case of a bounded perturbation $\vf$. Our approach is different as we use an approximation of $\she$ by partition functions of the Brownian semidiscrete directed polymer in the critical case $\alpha=\frac14$ rather than by particle systems.

Building on the techniques of \cite{AKQ0}, it is shown in \cite{MQR-M} that a suitable
renormalization of the partition function of the semi-discrete model with $\alpha=\frac14$ converges
to $\she^{\vf}$. More precisely, let 
$\vf_n(x)= \vf\big(\!-\tfrac{x}{\sqrt{n}}\big)$ and let 
\begin{align}\label{shevf3}   
 Z^{\para,\beta,\vf}_{n,t} &=
\int_{-\infty}^t   \exp\bigl[ \vf_n(s_0)-\beta B(s_0)+\beta\para s_0] \, 
  Z_{(1,n)(s_0,t)}(\beta) \,ds_{0}.
\end{align}
The renormalized partition function is
\begin{eqnarray*}   
\she^{\vf}_n(\tau) &=& e^{-\frac12 \tau \sqrt{n}}\, Z^{\para,\beta, \vf}_{\tau n,\tau n}.
\end{eqnarray*}
When $\vf=0$, we simply denote this by $\she_N(\tau)$.

\begin{theorem}\label{thm:monster}\cite{MQR-M}
 Let $\beta=\beta_n=n^{-1/4}$ and $\theta= \Psi^{-1}_1(\beta_n^2)$. 
Then as $n\to +\infty$, the process $(\she^{\vf}_n(\tau), \tau\ge 0)$ converges in law
to $(\she^{\vf}(\tau,0), \tau\ge 0)$, where $\she^{\vf}$ solves the stochastic heat equation
\eqref{shevf1}--\eqref{shevf2}.
\end{theorem}

Combined with Theorem \ref{thm-unscaled} this gives

\begin{theorem}\label{kpzclose}
Let $\vf$ be a bounded function and let $\she^{\vf}$ be the solution of the stochastic heat equation
\eqref{shevf1}--\eqref{shevf2}. Then there exist constants $C_1,\, C_2,\, \tau_0>0$ such that
\begin{eqnarray*}
 C_1\tau^{\frac23} \le \Vvv \log \she^{\vf}(\tau,0) \le C_2\tau^{\frac23},
\end{eqnarray*}
for all $\tau>\tau_0$.
\end{theorem}

We note that our results for the path of the stationary polymer could in principle have a meaning in the context of the SHE. In \cite{AKQ2}, $\she$ is identified as the partition function of a continuum directed polymer. Theorem \ref{thm:statpath} strongly suggests that the fluctuations of the path of the continuum polymer are of order $t^{2/3}$, in agreement with the KPZ scaling.


\section{Proofs for the semi-discrete polymer model}\label{pf-dpbe}
The proofs of our Theorems \ref{thm:freeZ}, \ref{thm:path-ptp}, \ref{thm-unscaled} and \ref{thm:statpath} are given in this section. We first prove the results for the stationary model. The results for the point-to-point model are then done by comparison.

%
%
%

\subsection{Preliminaries}\label{pf-dpbe-pre}
We recall some   facts from  \cite{SV}. Throughout this section we take  $\beta=1$ as we can  reduce the situation to this by Brownian scaling (see Section \ref{pf-dpbe-resc}). The stationary model can be  written as
\begin{eqnarray}\label{ZB2}
 Z^{\theta}_{n,t}=\int^t_0 e^{-B(s)+\theta s}Z_{(1,n),(s,t)}\, ds +
\sum^n_{j=1} \left(\;
\prod^j_{k=1} e^{r_k(0)}\right)Z_{(j,n),(0,t)} 
\end{eqnarray}
where the $r_k$ processes  are defined recursively in (\ref{sc-r}). 
Recall that the random variables $r_k(0)$ are i.i.d.~and $e^{-r_k(0)}$ has Gamma$(\theta)$ distribution. 

The processes $r_k$ and $Y_k$ give  space and time increments of the
partition function:
\be\label{Y-increment}\begin{aligned} 
 r_k(t)&=\log Z^{\theta}_{k,t}-\log Z^{\theta}_{k-1,t}, \\ 
 Y_k(s,t)&=Y_k(t)-Y_k(s)= \theta (t-s) - \log Z_{k,t} + \log Z_{k,s}.
\end{aligned}\ee
The appearance of polygamma functions in our results
is natural because of the identities
\be \E [r_k(t)]=-\Psi_0(\para) \quad\text{and}\quad  \Vvv [r_k(t)]=-\Psi_1(\para).
\label{rk}\ee
(See Section \ref{gtb}.)

From \eqref{Y-increment}, (\ref{rk}) and $Z^\theta_{0,t}=\exp(-B(t)+\theta t)$ one immediately gets
\be
\label{ev-ident}
 \E (\log Z_{n,t}^{\para})=-n \digamf(\para)+\para t. 
 \ee
The variance of $\log Z_{n,t}^{\para}$ was computed in Theorem 3.6 in \cite{SV}: 
\be
 \Vvv( \log Z^{\theta}_{n,t})=n \Psi_1(\theta)-t + 2 E^{\theta}_{n,t}(\sigma^+_0)=t-n \Psi_1(\theta) + 2 E^{\theta}_{n,t}(\sigma^-_0)=E^{\theta}_{n,t}\left|\sigma_0 \right|.
\label{var-ident}\ee
We will also need the following lemma from \cite{SV}: 
 \begin{lemma}\cite[Lemma 4.3]{SV}\label{4.1lem}
For  $\theta,  \lambda>0$,   
\[
\bigl\lvert \Vvv(\log Z^{\lambda}_{n,t}) - \Vvv(\log Z^{\para}_{n,t}) \bigr\rvert 
\le  n\abs{\trigamf(\lambda)-\trigamf(\para)}.  
\]
\end{lemma}
Finally, we note a shift-invariance of the stationary model (Remark 3.1
of \cite{SV}): 
\be
E^{Q^\para_{n,t}}f(\sigma_0,\sigma_1,\dotsc,\sigma_{n-1}) \eqd 
E^{Q^\para_{n,0}}f(t+\sigma_0,t+\sigma_1,\dotsc,t+\sigma_{n-1}).
\label{altQ2} \ee
This follows from the stationarity of $Z_n^\para(t) \exp(B(t)-\para t)$ by observing that the density of $(\sigma_0,\dots, \sigma_{n-1})$ under $Q^{\theta}_{n,t}$ can also be
written as
\be\begin{aligned}
\frac1{\widehat Z_n^\para(t)}
&\exp\bigl[
\widehat B(s_0,t)+\widehat B_1(s_0,s_1)+\dots+\widehat B_n(s_{n-1},t)
\bigr] \times \ind\{s_0<\dotsm<s_{n-1}<t\}
\end{aligned}\label{altQ}\ee
where $\widehat B(u)=B(u)-\para u/2$ (and similarly for $\widehat B_k$) 
and $\widehat Z_n^\para(t)=Z_n^\para(t) \exp(B(t)-\para t)$. 

Using the same ideas one can also show a shift-invariance property in $n$ (see the proof of Theorem 6.1 in \cite{SV}): 
\begin{align}\label{altQ3}
E^{Q^\para_{n,t}}f(\sigma_k,\sigma_{k+1},\dotsc,\sigma_{n-1}) \eqd 
E^{Q^\para_{n-k,t}}f(\sigma_0,\sigma_1,\dotsc,\sigma_{n-k-1}).
\end{align}

\subsection{Rescaled models and characteristic direction}\label{pf-dpbe-resc}
For the proofs we  scale $\beta$ away via the following identity in law which is obtained by Brownian scaling:   
\begin{eqnarray}\label{bscaling}
 Z_{(1,n),(0,t)}(\beta)\overset{d}= {\beta^{-2(n-1)}} Z_{(1,n),(0,\beta^2 t)}(1).
\end{eqnarray}
We drop  $\beta=1$ from the notation and write 
$Z_{(1,n),(0, t)}=Z_{(1,n),(0,t)}(1)$. 
The regime $\beta  = \betaa n^{-\alpha}$ corresponds to studying 
$Z_{(1,n),(0,\betaa^2n^{1-2\alpha})}$.        
Similarly we  scale $\beta$ away from the stationary partition function
\eqref{Zdef2.1}:  
\be
\label{statscaling}
Z^{\para,\beta}_{n,t}\eqd   \beta^{-2n} Z^{\beta^{-1}\para,1}_{n,\,\beta^2 t}. 
\ee 
 As we take $(n,t)$ to infinity, we have to follow approximately 
 a  characteristic direction determined by   $\theta$. The characteristic direction is found  by minimizing the right-hand side
of  \eqref{ev-ident} with respect to $\theta$, or equivalently, by arranging the 
cancellation of the first two terms on the right of 
\eqref{var-ident}. The following condition on the  triples $(n,t,\theta)$ expresses the fact that $(n,t)$ is close to the characteristic direction:   
\be   \abs{n\Psi_1(\theta)-t\,} \le \kappa n^{2/3}\theta^{-4/3} \qquad \textup{with  a fixed constant $\kappa\ge 0$}.
\label{nt-hyp}\ee

By the scaling relation \eqref{statscaling}, we can see that the choice of parameters in Theorem \ref{thm-unscaled} corresponds to the characteristic direction.


\subsection{Upper bounds for the stationary model}\label{pf-dpbe-ub-stat}

The main tool for our upper bounds is the following lemma. The proof of the upper bound in Theorem \ref{thm-unscaled} will follow by a particular choice of the parameters and can be found at the end of this section.

\begin{lemma}\label{main-upper-bound}   
Fix $\theta_0>0$ and $\kappa\ge 0$. 
Assume that $\theta>0$, $n\in\bN$ and $t>0$  satisfy \eqref{nt-hyp}  and $\theta_0\le \theta\le \theta_0^{-1} \sqrt{n}$. 
 Then there are constants $\delta>0$ and $c<\infty$ 
 that depend only on $\theta_0$ 
such that, for all 
$ n \theta^{-1}\ge u\ge 2 \kappa n^{2/3}\theta^{-4/3}$, we have
\begin{align}
\bP\big\{ Q^{\theta}_{n,t}(\sigma^{+}_0 \geq u) \geq e^{-\delta \theta^2u^2n^{-1}}\big\}
&\le 
c(1+\kappa) \frac{n^{8/3}}{\theta^{16/3}u^4}  
+c  \frac{n^{2}}{\theta^4u^3}  , 
\label{aa-1}\\
 E^{\theta}_{n,t}(\sigma^{+}_0)& \le  c(1+\kappa) \frac{n^{2/3}}{\theta^{4/3}} \label{aa-2},\\
\text{and} \qquad   P^{\theta}_{n,t}\bigl\{\sigma^+_0\geq b n^{\frac23} \theta^{-\frac43}\bigr\}&\leq  
c(1+\kappa) b^{-3} 
\qquad\text{for $b\ge (2\kappa)\vee 1$.}   \label{aa-3}
\end{align}
The same bounds hold for $\sigma_0^{-}$, and for $u\ge n \theta^{-1}$ we also have
\begin{align}
\label{aa-3.5}
\bP\big\{ Q^{\theta}_{n,t}(\sigma^{-}_0 \geq u) \geq e^{-\delta \theta u}\big\}
&\le 
2 e^{-c \theta u}. 
\end{align}
Finally,  we also have
\begin{align}
\label{aa-4}
\Vvv (\log Z_{n,t}^\theta)\le c(1+\kappa) n^{2/3} \theta^{-4/3}. 
\end{align}
\end{lemma}
\begin{proof} We introduce $a=\delta  \theta^2u^2n^{-1}$. We fix the positive parameter $r$ (its value will be determined later), 
and  set  $\lambda= \theta+{r u\theta^2}{n}^{-1}$. 
From the definition of the path measure, we have 
\begin{align}
Q^\para_{n,t}(\sigma_0^+\ge u)&=\frac1{Z^{\para}_{n,t} } 
\int\limits_{u<s_{0}<\dotsm<s_{n-1}<t} 
\exp\bigl[ -B(s_0)+\para s_0 +  B_1(s_0,s_1) 
+\dotsm + B_n(s_{n-1},t)\bigr] \,ds_{0,n-1}\nn\\[4pt]
&\le\frac1{Z^{\para}_{n,t} } \label{ineq-sigma0}
\int\limits_{u<s_{0}<\dotsm<s_{n-1}<t}  \hskip-20pt e^{(\para-\lambda)u}
  \exp\bigl[ -B(s_0)+\lambda s_0 +  B_1(s_0,s_1) 
+\dotsm + B_n(s_{n-1},t)\bigr] \,ds_{0,n-1}\\[4pt]
&\le \frac{Z^{\lambda}_{n,t}}{Z^\para_{n,t} }e^{(\para-\lambda) u}.\nn
\end{align}
Consequently,
\begin{align}
\bP\big\{ Q^{\theta}_{n,t}(\sigma^+_0 \geq u) \geq e^{-a}\big\}\nn&\leq  \bP\big\{ \log Z^{\lambda}_{n,t}-\log Z^{\theta}_{n,t}\geq (\lambda-\theta)u-a\big\} \nn\\[5pt]
&= \bP\big\{ \overline{\log Z^{\lambda}_{n,t}}-\overline{\log Z^{\theta}_{n,t}}\geq
n(\Psi_0(\lambda)-\Psi_0(\theta))-t(\lambda-\theta)+(\lambda-\theta)u-a\big\},
\label{line-a-5}\end{align}
where   $\overline  X=X-EX$ denotes the centering of the random variable $X$.    Because of (\ref{ev-ident}) we have $\overline{\log
Z^{\lambda}_{n,t}}=\log Z^{\lambda}_{n,t}+n \Psi_0(\theta)-\theta t$.

By the monotonicity of $\Psi_2(z)=\Psi_0''(z)$ (see \eqref{Psi-2}) 
for any $\lambda>\theta>0$  
\[  0\ge  \Psi_0(\lambda)-\Psi_0(\theta)- \Psi_1(\theta)(\lambda - \theta)=\int_\theta^\lambda \int_\theta^y \Psi_2(z) dz dy
\ge -\tfrac12 \abs{\Psi_2(\theta)} (\lambda - \theta)^2. \]

Assumptions  (\ref{nt-hyp})  and $ u\ge 2 \kappa n^{2/3}\theta^{-4/3}$  imply
\be \abs{n\Psi_1(\theta)-t\,} \le u/2 \label{line-a-5.1}\ee
and so the  right-hand side inside the probability (\ref{line-a-5})  develops as follows: 
\begin{align}\nn
&n(\Psi_0(\lambda)-\Psi_0(\theta))-t(\lambda-\theta)+(\lambda-\theta)u-a\\
&\nn\hskip70pt =n \bigl( \Psi_0(\lambda)-\Psi_0(\theta)-\Psi_1(\theta)(\lambda - \theta)\bigr)+(n
\Psi_1(\theta)-t)(\lambda-\theta)+u(\lambda-\theta)-a\\
&\hskip70pt \ge -\frac{n}2\abs{\Psi_2(\theta)} (\lambda - \theta)^2
+\frac{u}2(\lambda-\theta)-a\ge \left(-\frac{r^2 c_0}2+\frac{r}{2}-\delta\right)\theta^2u^2n^{-1}\nn\\
&\hskip70pt \ge   \delta \theta^2u^2n^{-1}. \nn  \label{line-a-5.2}
\end{align}
Above we introduced   
\be c_0=\sup_{x\ge \theta_0 }|\Psi_2(x)|x^2<\infty \label{c0}\ee
 (which is finite by \eqref{Psi-2} and \eqref{Psi-3}) and then chose 
$r=(2 c_0)^{-1}$ and $ \delta=c_0^{-1}/16$.

In the following   $c$ denotes 
a constant that  depends only on  $\theta_0$, but may change from line to line. 
From line \eqref{line-a-5}, using Lemma \ref{4.1lem} we get
\be\begin{aligned}
\bP\big\{ Q^{\theta}_{n,t}(\sigma^+_0 \geq u) \geq e^{-a}\big\}&\leq \bP\big\{  \overline{\log Z^{\lambda}_{n,t}}-\overline{\log Z^{\theta}_{n,t}}\geq \delta\theta^2u^2n^{-1} \big\}\\
&\leq c \frac{n^{2}}{\theta^4u^4} \Vvv\bigl[ \log Z^{\lambda}_{n,t}-\log Z^{\theta}_{n,t}\bigr]\\
&\leq c \frac{n^{2}}{\theta^4u^4} \bigl( \,
\Vvv\bigl[\log Z^{\theta}_{n,t}\bigr] + n |\Psi_1(\lambda)-\Psi_1(\theta)|\,\bigr)\\
&\leq c \frac{n^{2}}{\theta^4u^4}  \bigl(  E^{\theta}_{n,t}(\sigma^+_0) +
u \bigr). 
\end{aligned}\label{line-a-5.3}\ee
Above we used \eqref{var-ident}, \eqref{line-a-5.1}, and the following estimate: 
  \[ |\Psi_1(\lambda)-\Psi_1(\theta)|\leq |\Psi_2(\theta)|(\lambda-\theta) 
\le c_0 \theta^{-2} (\lambda-\theta)= c_0 ru/n=u/(2n).\]
Let $u_0\ge 2\kappa n^{2/3}\theta^{-4/3}$.  
\begin{align*}
 E^{\theta}_{n,t}(\sigma^+_0) &\leq u_0 + \int^t_{u_0} du\, 
P^{\theta}_{n,t}\left[ \sigma^+_0 \geq u\right]\\
 &\leq u_0 + \int^t_{u_0} du\, \Bigl\{ 
\int^1_{e^{-a}} dr\, \bP\bigl[ Q^{\theta}_{n,t}(\sigma^+_0 \geq u) \ge r \bigr]
+ e^{-a} \Bigr\} \\
 &\leq u_0 + \frac{cn^2}{\theta^4}  \int^t_{u_0} du\, 
  \left( \frac{ E^{\theta}_{n,t}(\sigma^+_0)}{u^4} + \frac{1}{u^ 3} \right)  + 
\int^t_{u_0} e^{-\delta \theta^2 u^2 /n }\,du \\
 &\leq u_0 + \frac{cn^2}{\theta^4u_0^3} E^{\theta}_{n,t}(\sigma^+_0) 
+  \frac{cn^2}{\theta^4u_0^2}
 +  \frac{\delta^{-1} n}{ 2\theta^2u_0} e^{-\delta  \theta^2 u_0^2 /n }. 
\end{align*} 
The last term comes from $\int_m^\infty e^{-x^2}\,dx\le (2m)^{-1}e^{-m^2}$ for $m>0$. 
Now choose   $u_0=  2(1+c+\kappa) n^{2/3}\theta^{-4/3}$.
The inequality above can be rearranged
to give 
\be \begin{aligned}
 E^{\theta}_{n,t}(\sigma^+_0) &\le   (c+4\kappa) \frac{n^{2/3}}{\theta^{4/3}} 
 +  \frac{cn^{1/3}}{\theta^{2/3}} \exp(-\delta n^{1/3}  \theta^{-2/3}   )\\
 &\le   c(1+\kappa) \frac{n^{2/3}}{\theta^{4/3}} .
\end{aligned}\label{line-a-6}  
\ee 
Above, $c$ has been redefined but still depends only on $\theta_0$. This proves (\ref{aa-2}) for $\sigma_0^+$. 
Substitute this back up in \eqref{line-a-5.3}  to get 
\be
\bP\big\{ Q^{\theta}_{n,t}(\sigma^+_0 \geq u) \geq e^{-\delta \theta^2u^2n^{-1}}\big\}
\le 
c(1+\kappa) \frac{n^{8/3}}{\theta^{16/3}u^4}  
+c(1+u)  \frac{n^{2}}{\theta^4u^4}      
\label{line-a-7}\ee
which proves (\ref{aa-1}) as $\theta\le \theta_0 \sqrt{n}$ . To prove (\ref{aa-3}) 
apply (\ref{aa-1}) with $u=b n^{2/3} \theta^{-4/3}$, and use $b\ge (2\kappa)\vee 1$:
\begin{align}
  P^{\theta}_{n,t}\bigl\{\sigma^+_0\geq b n^{\frac23} \theta^{-\frac43}\bigr\}
  &\le e^{-\delta \theta^2u^2n^{-1}}+\bP\big\{ Q^{\theta}_{n,t}(\sigma^+_0 \geq b n^{\frac23} \theta^{-\frac43}) \geq e^{-\delta \theta^2u^2n^{-1}}\big\}\nn\\\label{aaa1}
&\le e^{-\delta \theta^{-2/3}n^{1/3}b^2}+  c(1+\kappa) b^{-4} +cb^{-3} \\
&\le  c(1+\kappa) b^{-3}   .\nn 
\end{align}
The proof for $\sigma^-_0$ is similar, but we need some modifications. We take $\lambda=\theta-{r u\theta^2}{n}^{-1}$. Note that by choosing $r<1/2$ and using $n \theta^{-1} \ge u$ we have $\lambda>\theta/2$. 
Now we use the inequality
\[
Q^{\theta}_{n,t}(\sigma^-_0 \geq u) \leq
\frac{Z^{\lambda}_{n,t}}{Z^{\theta}_{n,t}}e^{-(\theta-\lambda)u},
\]
instead of (\ref{ineq-sigma0}), its proof being similar. Bound (\ref{aa-1}) now follows exactly the same way as for $\sigma_0^+$. 

In order to get the bound (\ref{aa-3.5}) for $u\ge n \theta^{-1}$ we set $\lambda=r \theta$ with $r$ to be specified later and proceed with the proof exactly the same way as in the previous case. We have
\begin{align*}
\bP\big\{ Q^{\theta}_{n,t}(\sigma^{-}_0 \geq u) \geq e^{-\delta \theta u}\big\}
&\\
&\hskip-40pt \le \bP\big\{ \overline{\log Z^{\lambda}_{n,t}}-\overline{\log Z^{\theta}_{n,t}}\geq
n(\Psi_0(\lambda)-\Psi_0(\theta))-t(\lambda-\theta)+(\theta-\lambda)u-\delta \theta u\big\}.
\end{align*}
and
\begin{align*}
&n(\Psi_0(\lambda)-\Psi_0(\theta))-t(\lambda-\theta)+(\lambda-\theta)u-\delta \theta u \\
&\nn\hskip70pt =n \bigl( \Psi_0(\lambda)-\Psi_0(\theta)-\Psi_1(\theta)(\lambda - \theta)\bigr)+(n
\Psi_1(\theta)-t)(\lambda-\theta)+u(\theta-\lambda)-\delta \theta u\\
&\hskip70pt \ge -\frac{n}2\abs{\Psi_2(\lambda)} (\lambda - \theta)^2
+\frac{u}2(\theta-\lambda)-\delta \theta u\ge
-n C(\theta_0) r^2+\frac14 r u \theta-\delta \theta u
\ge
 c_0  \theta u \end{align*}
with a fixed positive $c_0$. In order to get the last bound we need to choose $r$ and $\delta$  small enough in terms of $c(\theta_0)$.
This gives
\begin{align*}
\bP\big\{ Q^{\theta}_{n,t}(\sigma^{+}_0 \geq u) \geq e^{-\delta \theta u}\big\}&\le
 \bP\big\{ 
\overline{\log Z^{\lambda}_{n,t}}-\overline{\log Z^{\theta}_{n,t}}\ge c_0 \theta u\big\}\\
&\le
  \bP\big\{ 
\overline{\log Z^{\lambda}_{n,t}}+B(t)\ge c_0 \theta u/2\big\}+ 
\bP\big\{ 
\overline{\log Z^{\theta}_{n,t}}+B(t)\ge c_0 \theta u/2\big\}
\end{align*}
By (\ref{Zr-ident})  $\log Z_{n,t}^\theta+B(t)-\theta t$ is the sum of $n$ i.i.d.~Gamma($\theta$) variables. Using the exponential Markov inequality (optimizing in the extra parameter) we get for any $y>0$ that
\[
\bP(\overline{\log Z^{\theta}_{n,t}}+B(t) \ge n \theta y)\le e^{- n \theta (y-\log(1+y))}.
\]
We have $y-\log(1+y)>K(C) y$ if $y>C$. Since $u>C n \theta^{-1}>C' n$ we have the bound $c_0  u n^{-1}/2>C_1(c_0,C')$ and from this we get
\[
\bP(\overline{\log Z^{\theta}_{n,t}}+B(t) \ge  \theta c_0 u/2)\le e^{- c_1 \theta u}.
\]
and a similar bound for the $\lambda$ term. This completes the proof of (\ref{aa-3.5}) for $\sigma_0^-$. 

To prove (\ref{aa-2}) for $\sigma_0^{-}$ we use $\E_{n,t}^\theta \sigma_0^+-\sigma_0^-=t-n \Psi_1(\theta)$ and the fact that we already have (\ref{aa-2}) for $\sigma_0^{+}$.  To prove (\ref{aa-3}) for $\sigma_0^-$ we can follow  (\ref{aaa1}) in the case $b n^{\tfrac23} \theta^{-\tfrac43}\le n \theta^{-1}$ and use the bound  (\ref{aa-3.5}) with a similar argument if $b n^{\tfrac23} \theta^{-\tfrac43}\ge n \theta^{-1}$.

Finally, bound (\ref{aa-4}) follows from (\ref{var-ident}),
\eqref{nt-hyp}  and (\ref{aa-2}). 
\end{proof}


\begin{proof}[Proof of  the upper bound in Theorem \ref{thm-unscaled}]
Introduce variables 
\be \text{$\tilde n=\tau n$,  $t=\tau \beta_0^2 n^{1-2\alpha}$ and $\tilde \theta=\Psi_1^{-1}(\beta_0^2 n^{-2\alpha})$. }
\label{var7}\ee
By the scaling identity (\ref{statscaling}),  $\Vvv (\log Z_{\tau n,\tau n}^{\para, \beta})=\Vvv (\log Z_{\tilde n,t}^{\tilde \para})$. Condition  (\ref{nt-hyp}) is satisfied by $(\tilde n, t,  \tilde \theta)$ with $\kappa=0$, 
$\tilde \theta\ge \Psi_1^{-1}(\beta_0^2)>0$ and $\tilde \theta\le C \beta_0^{-2} n^{2\alpha}\le C'  \sqrt{\tilde n}$, as long as $\tau\geq \tau_0$ for a constant $\tau_0=\tau_0(\beta_0)$. 
This means that we may apply Lemma \ref{main-upper-bound} with $(\tilde n, t, \tilde \theta)$. The bound (\ref{aa-4})  gives
\[
 \Vvv\bigl[ \log Z^{\tilde \theta}_{\tilde n, t}\bigr]\le c\frac{\tilde n^{2/3}}{\tilde \theta^{4/3}}\le C \tau^{\frac23} n^{\frac23(1-4\alpha)}
\]
where $C$ depends only on $\beta_0$.
\end{proof}

\subsection{Lower bound for the stationary model}\label{pf-dpbe-lb-stat}
In this section we prove the lower bound in Theorem \ref{thm-unscaled}. Again, the proof will follow by a particular choice of the parameters in the next proposition and can be found at the end of the section. 

\begin{proposition}\label{prop-lower-bd}
Let $\theta_0>0$ and $\kappa\ge 0$.   
There are positive constants $\delta_1, \delta_2, n_0$ 
that depend on $(\kappa, \theta_0)$ such that 
   \begin{align}
   \label{Prop4LBdist}
\P\bigl\{ \,{\log Z_{n,t}^\theta}-\E ({\log Z_{n,t}^\theta})\ge \delta_1
n^{\frac13}\theta^{-\frac23}\bigr\}\geq \delta_2  
 \end{align}
whenever  $n\ge n_0$,  $(n, t, \theta)$ satisfies \eqref{nt-hyp},  
and $\theta_0\le \theta \le \theta_0^{-1} \sqrt{n}.$

Moreover, under the previous assumptions we also have
\begin{align}
\label{Prop4LBvar}
\Vvv\bigl[ \log Z^{ \theta}_{ n, t}\bigr]\ge c { n^{\frac23}}{ \theta^{-\frac43}}
\end{align}

\end{proposition}

\begin{proof}
It is sufficient to prove  estimate (\ref{Prop4LBdist}) since the lower bound (\ref{Prop4LBvar}) follows from this easily.
Fix a  constant $0<b<\theta^{1/3} n^{1/3}$ 
and set 
  $\lambda= \theta + b \theta^{2/3}n^{-1/3}<2\theta$ and $\bar t=t+n\trigamf(\lambda)-
n\trigamf(\theta)$. Then we have 
\begin{align}\label{lower0}
v=t-\bar t =n(\trigamf(\theta) -\trigamf(\lambda) ) &\ge  n |\Psi_2(\lambda)|(\theta-\lambda)\ge
 4^{-1} b n^{\frac23}\theta^{-\frac43}
\end{align}
where we used $|\Psi_2(\lambda)|\ge \lambda^{-2}\ge \theta^{-2}/4 $
from \eqref{Psi-3}.  
 We shall take $b\in(0,\infty)$ large enough
in the course of the argument, which is not problematic as $\theta^{1/3} n^{1/3}$ will be large for large enough $n$ by our assumption $\theta\ge \theta_0$. 

Fix a  $c_1\in(0,1/2)$. By the  
shift-invariance (\ref{altQ2}) we have
\begin{align*}
Q_{n,t}^\lambda(\sigma_0^{+}\le c_1v )&=Q_{n,t}^\lambda(\sigma_0 \le c_1v )\\
&\eqd Q_{n,\bar t}^\lambda(t-\bar t+\sigma_0 \le c_1v)=Q_{n,\bar t}^\lambda(\sigma_0 \le -(1-c_1)v)\\
&=Q_{n,\bar t}^\lambda(\sigma_0^{-} \ge (1-c_1)v). 
\end{align*}
Since
$\lambda\le 2\theta$,   $(n, \bar t,\lambda)$ satisfies  \eqref{nt-hyp} with  
$\kappa$ replaced by $2^{4/3}\kappa$.  
We can  apply the upper bound (\ref{aa-1}) to $\sigma^-$ with $(n, \bar t,\lambda)$ and $u=(1-c_1) v$ because 
\[ 
(1-c_1)v \ge \frac18  b n^{\frac23} \theta^{-\frac43}\ge  2 (2^{4/3}\kappa)  \lambda^{-4/3} n^{\frac23}
 \]
if we  choose $b>2^{4+4/3}\kappa$. After collecting all the terms on the right of (\ref{aa-1}) we get the upper bound
\[
P(Q_{n,t}^\lambda(\sigma_0^{+}\le c_1v )>\eps_0)\le Cb^{-3}
\]
for any fixed $\eps_0>0$ if $b$ is large enough (and hence also $n$) relative to $\eps_0$ and $\theta_0$.  This choice of $b$ is needed to ensure that
$\eps_0 > e^{-\delta \lambda^2 u^2 n^{-1}}$.

A similar shifting argument gives 
 \[
 Q_{n,t}^\lambda(\sigma_0^{+}\ge (1+c)(t-\bar t))\eqd  Q_{n,\bar t}^\lambda(\sigma_0^{+}\ge c
(t-\bar t)), 
 \]
 and the upper bound (\ref{aa-1}) can be applied with $\lambda$ and $\bar t$.
  Hence, we can fix constants $0<c_1<c_2<\infty$ such that, for a given $\e_0>0$ and 
  large $n$, 
 \be \P\bigl[\,   Q^{\lambda}_{n,t }\{c_1 v\le \sigma_0 \le  c_2 v\}  \ge  1-\e_0 \bigr]
\; \ge \; 1-Cb^{-3}.  \label{lower1}
\ee
Observe that $c_1$ can be taken as close to 0 as we wish.

Introduce  temporary notations $s_1=c_1 v$ and  $s_2=c_2 v$. Assumptions
\eqref{nt-hyp}  and 
$\theta\ge \theta_0>0$ guarantee that $s_2<t$ for large enough $n$. Hence
\[
Q^{\lambda}_{n,t }\{s_1 \le \sigma_0 \le  s_2\}=\frac{1}{Z_{n,t}^\lambda} \int_{s_1}^{s_2}
e^{-B(s)+\lambda s} Z_{1,n}(s,t) ds . 
\]
Using \eqref{nt-hyp}, 
\begin{align*}
\E (\log Z^{\lambda}_{n,t })-\E (\log Z^{\theta}_{n,t })&=n
(\digamf(\theta)-\digamf(\lambda))+(\lambda-\theta)t\\
&\geq n\int_\theta^\lambda (\trigamf(\theta)-\trigamf(\xi))\,d\xi - \kappa
n^{2/3}\theta^{-4/3}(\lambda-\theta) \\
 &\ge \tfrac12 n \abs{\Psi_2(\lambda)}  (\lambda-\theta)^2  - \kappa b \theta^{-2/3}
n^{1/3}\\
&\ge \tfrac{1}{8}n^{1/3}b^2 \theta^{-2/3}-  \kappa b \theta^{-2/3}
n^{1/3}
\ge 2c_* b^2 n^{\frac13} \theta^{-\frac23}, 
\end{align*}
with the constant $c_*=\frac{1}{32}$ if we choose $b\ge {16}\kappa$. 
Then from (\ref{lower1}) 
\begin{align}
 1-Cb^{-3} &\le  \P\Bigl[\,   Q^{\lambda}_{n,t }\{s_1 \le \sigma_0 \le  s_2 \}  \ge  1-\e_0 \Bigr]\nn\\
&\le \P\Bigl[ \, \int_{s_1}^{s_2} e^{-B(s)+\lambda s} Z_{1,n}(s,t) ds 
\ge (1-\eps_0) e^{\E (\log Z^{\theta}_{n,t })+c_*  b^2 n^{\frac13}\theta^{-\frac23}}
\,\Bigr] \label{ub-86}\\
&\qquad+\P\bigl(Z_{n,t}^\lambda \le e^{\E (\log Z^{\lambda}_{n,t })-c_*b^2
n^{\frac13}\theta^{-\frac23}}\bigr).  \label{ub-87}
 \end{align}
Bound probability \eqref{ub-87} with Chebyshev:
 \be\nn\label{lower2}
 \P\bigl(Z_{n,t}^\lambda \le e^{\E (\log Z^{\lambda}_{n,t })-{c_*}b^2
n^{\frac13}\theta^{-\frac23}}\bigr)
 \le c_*^{-2} b^{-4} \theta^{\frac43}n^{-\frac23} \Vvv(\log Z_{n,t}^\lambda). 
 \ee
 The variance is estimated by (\ref{4.1lem}) and (\ref{lower0}):
 \[
\Vvv(\log Z_{n,t}^\lambda)\le \Vvv(\log Z_{n,t}^\theta)+n \abs{\trigamf (\lambda)-\trigamf(\theta)}
\le C (1+b)n^{\frac23}\theta^{-\frac43}. 
 \]
 This implies  that \eqref{ub-87} $\le$ $C b^{-3}$. 
 
Let $A$ denote the event in probability \eqref{ub-86}:
 \begin{align*}
\P(A)&= \P\Bigl[ e^{-B(s_1)+\lambda s_1} \int_{s_1}^{s_2} e^{-B(s_1,s)+\lambda (s-s_1)} Z_{1,n}(s,t) ds 
\ge (1-\eps_0) e^{\E (\log Z^{\theta}_{n,t })+{c_*}  b^2 n^{\frac13}\theta^{-\frac23}}\Bigr] \\
&= \P\Bigl[ e^{-B(s_1)+\theta s_1} \int_{s_1}^{s_2} e^{-B(s_1,s)+\lambda (s-s_1)} Z_{1,n}(s,t)
ds
\ge (1-\eps_0) e^{\E(\log Z^{\theta}_{n,t })+c_*  b^2
n^{\frac13}\theta^{-\frac23}-(\lambda-\theta)s_1}\Bigr]  
 \end{align*}
 We wish to  replace $\lambda$ by $\theta$ inside the integral to match it with the parameter in   $\E(\log Z^{\theta}_{n,t })$ on the right-hand side. 
For this  we  use the Cameron-Martin-Girsanov formula to add a drift $\lambda-\theta$ to the
 Brownian motion $\{ B(s_1,s): s_1\le s \le s_2\}$.  
 Note that the other random objects in the event 
 $A$, namely   $\{B(s_1);  B_i(\cdot): 1\le i\le n\}$,
    are independent of $B(s_1,\cdot\,)$.  
 Let 
 \[
 \frac{d\wt\P}{d\P} =e^{(\lambda-\theta)B(s_1,s_2)-\frac12 (\theta-\lambda)^2 (s_2-s_1)}
 \]
so that, under $\wt\P$,  $B(s_1,s)\eqd \tilde B(s_1,s)+(\lambda-\theta)(s-s_1)$ where  $\tilde B$ is a 
   standard Brownian motion.
By   Cauchy-Schwarz
 \[
 \P(A)=\wt\E\Bigl[ \frac{d \P}{d\wt\P} \ind(A)\Bigr]\le \sqrt{ \wt\E \Bigl[\Bigl( \frac{d
\P}{d\wt\P} \Bigr)^2\;\Bigr]}\sqrt{\wt\P (A)}. 
 \]
The first expectation is finite: 
 \begin{align*}
  \wt\E \Bigl[ \Bigl(\frac{d \P}{d\wt\P}\Bigr)^2\;\Bigr]&=\wt\E e^{2(\theta-\lambda)B(s_1,s_2)+(\theta-\lambda)^2
(s_2-s_1)}=\wt\E e^{2(\theta-\lambda)(\tilde
B(s_1,s_2)+(\lambda-\theta)(s_2-s_1))+(\theta-\lambda)^2 (s_2-s_1)}\\
  &=\wt\E e^{2(\theta-\lambda)\tilde B(s_1,s_2)- (\theta-\lambda)^2 (s_2-s_1)}=e^{(\theta-\lambda)^2
(s_2-s_1)}\le e^{C b^3}
 \end{align*} 
 where $C$ depends only on $\theta_0$. 
 We bound the probability $\wt\P(A)$ as follows: recall $c_0$ from \eqref{c0}, 
  \begin{align*}
 \wt\P (A) &=\wt\P\Bigl[ e^{-B(s_1)+\theta s_1} \int_{s_1}^{s_2} e^{-B(s_1,s)+\lambda (s-s_1)} Z_{1,n}(s,t)
ds 
\ge (1-\eps_0) e^{E \log  Z^{\theta}_{n,t }+c_*  b^2
n^{\frac13}\theta^{-\frac23}-(\lambda-\theta)s_1}\Bigr]\\
&=\wt\P\Bigl[e^{-B(s_1)+\theta s_1} \int_{s_1}^{s_2} e^{-\tilde B(s_1,s)+\theta (s-s_1)}
Z_{1,n}(s,t) ds 
\ge (1-\eps_0) e^{E \log Z^{\theta}_{n,t }+c_*  b^2
n^{\frac13}\theta^{-\frac23}-(\lambda-\theta)s_1}\Bigr]\\
&\le  \P\bigl(Z_{n,t}^\theta
\ge (1-\eps_0) e^{E \log Z^{\theta}_{n,t }+c_* b^2
n^{\frac13}\theta^{-\frac23}-(\lambda-\theta)s_1}\bigr)\\
&\le  \P\bigl(\,\overline{\log Z_{n,t}^\theta}\ge \log (1-\eps_0)+c_* b^2
n^{\frac13}\theta^{-\frac23}-c_0c_1 b^2 n^{\frac13}\theta^{-\frac23}\bigr)\\
&\le  \P\bigl(\,\overline{\log Z_{n,t}^\theta}\ge \tfrac1{2} c_* b^2 n^{\frac13}\theta^{-\frac23}\bigr)
 \end{align*}
 where the last line follows after choosing $c_1$    small enough.  
 
Put the estimates back on lines  \eqref{ub-86}--\eqref{ub-87} to conclude that  
  \[
(1-C b^{-3})^2 \exp(-C b^3)\le   \P\bigl\{ \,\overline{\log Z_{n,t}^\theta}\ge \tfrac1{2} c_* b^2
n^{1/3}\theta^{-\frac23}\bigr\}  
 \]
for a  constant $C$  that depends only on $(\kappa, \theta_0$). This completes the proof of the proposition. 
\end{proof}


\begin{proof}[Proof of the lower bound in Theorem \ref{thm-unscaled}]
As in the upper bound proof, condition  (\ref{nt-hyp})  with $\kappa=0$ is satisfied by 
variables $(\tilde n, t,\tilde \theta)$ from \eqref{var7}.  
By \eqref{Psi-2} and   \eqref{Psi-3}, for any $0<y_0<\infty$ there exists
a constant $C(y_0)$ such that 
  $\Psi_1^{-1}(y)\le C(y_0)/y$ for 
  $y\in(0, y_0]$.   Hence if $\alpha>0$, $\tilde \theta$
satisfies $0< \Psi_1^{-1}(\beta_0^2)\le \tilde \theta\le C_1\beta_0^{-2} n^{2\alpha}\le C_2 \sqrt{\tilde n}$ if $\tau\geq \tau_0$ for some positive constant $\tau_0(\beta_0)$, and Proposition  \ref{prop-lower-bd} can be applied.  
If $\alpha=0$   the hypothesis of Proposition  \ref{prop-lower-bd}  
is immediately true.  By the scaling identity (\ref{statscaling}) this proposition gives, with a constant $c>0$ that depends on $\beta_0$, 
\[  
\Vvv (\log Z_{\tau n,\tau n}^{\para, \beta})=
\Vvv(\log Z_{\tilde n,t}^{\tilde \theta})\ge c {\tilde n}^{2/3}{\tilde \theta}^{-4/3}\ge c \tau^{2/3} { n}^{\frac23(1-4\alpha)}
\qedhere \] 
\end{proof}
\subsection{Bounds on the path for the stationary model}
\label{pf-dpbe-stat-path}

\begin{proof}[Proof of Theorem \ref{thm:statpath}]
We start with the proof of the upper bound (\ref{statpathUB}). 
We introduce the familiar rescaling $\tilde \theta=\Psi_1^{-1}(\beta_0^2 n^{-2\alpha})$, $\tilde n=\tau n$ and $t=\tau \beta_0^2 n^{1-2\alpha}$. Then
\be\begin{aligned}
&P^{\para,\beta}_{\tau n,\tau n}\left\{ |\sigma_{\gamma \tau n}-\gamma \tau n|>b
\tau^{\frac23}n^{\frac23(1-\alpha)}\right\}
\le  P^{\tilde\para}_{\tilde n,t}\left\{ |\sigma_{\gamma\tilde n}-\gamma t|>b \beta_0^{-\frac23} {\tilde n}^{\frac23}{ \tilde \theta}^{-\frac43}  \right\}\\
&\qquad\qquad 
=  P^{\tilde\para}_{(1-\gamma)\tilde n,(1-\gamma)t}\left\{ |\sigma_0|>b \beta_0^{-\frac23} {\tilde n}^{\frac23} {\tilde \theta}^{-\frac43}  \right\} \le C b^{-3}.
\end{aligned}\label{ub-9}\ee
The first inequality is from   Brownian scaling (\ref{statscaling}). 
After the change of variable from the Brownian scaling step, 
the quantity on the right of the inequality inside the braces develops as follows:
\[   \beta^2 b\tau^{\frac23}n^{\frac23(1-\alpha)}
=  \beta_0^2 b\tau^{\frac23}n^{\frac23}  n^{-\frac83\alpha} 
\ge  b \beta_0^{-\frac23} {\tilde n}^{\frac23}{ \tilde \theta}^{-\frac43} 
\]
using   $\Psi_1^{-1}(x)\ge x^{-1}$ \eqref{Psi-3}.  
The second step (equality) in \eqref{ub-9} comes from   shift invariance in $t$ and $n$: (\ref{altQ2}) and (\ref{altQ3}).
The  last inequality  in \eqref{ub-9}  is  the upper bound (\ref{aa-3}) for $\sigma_0^\pm$. Note
that   $((1-\gamma)\tilde n, (1-\gamma)t, \tilde  \theta)$ satisfies  (\ref{nt-hyp}) with $\kappa=0$.  The constant $c$ from (\ref{aa-3})  depends  only on the
lower bound  $\tilde\theta_0=\Psi_1^{-1}(\beta_0^2)$ and so $C$ above
 depends only on $\beta_0$. 
This completes the proof of (\ref{statpathUB}).

We now prove the bound (\ref{statpath3}):
using  Brownian scaling (\ref{statscaling}) and   shift invariance (\ref{altQ2}), (\ref{altQ3}) again, 
\[
E^{\para,\beta}_{\tau n,\tau n} |\sigma_{\gamma \tau n}-\gamma \tau n|=\beta^{-2} E^{\tilde \para}_{\tilde n, t} |\sigma_{\gamma \tilde n}-\gamma t|=
\beta^{-2} E^{\tilde \para}_{(1-\gamma)\tilde n, (1-\gamma)t} |\sigma_{0}|. 
\]
By (\ref{var-ident})  
\[
\beta^{-2} E^{\tilde \para}_{(1-\gamma)\tilde n, (1-\gamma)t} |\sigma_{0}|=\beta_0^{-2} n^{4\alpha} \Vvv [\,\log Z^{\tilde \para}_{(1-\gamma) \tilde n (1-\gamma) t}\,].
\]
Now using (\ref{aa-4}) and (\ref{Prop4LBvar}) we have 
\[
C_1^{-1} (1-\gamma)^{2/3} \tilde n^{2/3} \tilde \theta^{-4/3} \le \Vvv \log [\,Z^{\tilde \para}_{(1-\gamma) \tilde n (1-\gamma) t}\,]\le C_1 ((1-\gamma)^{2/3} \tilde n^{2/3} \tilde \theta^{-4/3}+1)
\]
Using the asymptotics for $\tilde \para$ we get
\begin{align*}
C_1^{-1}  \tau^{2/3} n^{\frac23(1-\alpha)}\le E^{\para,\beta}_{\tau n,\tau n} |\sigma_{\gamma \tau n}-\gamma \tau n|&\le C_1 \tau^{2/3} n^{\frac23(1-\alpha)}.\qedhere
\end{align*}
\end{proof}


\subsection{Bounds for the point-to-point model}
\label{sec:freeZ}
This section derives  bounds on the path and  free energy fluctuations
in the point-to-point case without boundaries, with $\beta=1$, uniformly in $(n, t, \theta)$. Theorems \ref{thm:freeZ}   and  \ref{thm:path-ptp}   
 follow after a  Brownian scaling step.
For $n\in\bN$, $t>0$  and events $D$ on the paths 
  write $Z_{n,t}^\para(D)=Z_{n,t}^\para Q_{n,t}^\para(D)$ for the unnormalized
quenched measure.

\begin{theorem}\label{thm:freeZa}  Fix $0<\theta_0<\infty$. 
Let $\theta=\Psi_1^{-1}(t/n)$ satisfy 
$
\theta_0\le \theta \le \theta_0^{-1} \sqrt{n}.
$
 Then there exist constants $b_0, C$ that depend only on $\theta_0$ so that for $n\ge n_0$, $b\ge b_0$ we have 
\be \label{free0.1}
\P\bigl\{\abs{\log Z_{(1,n),(0,t)} - (\theta t-n \Psi_0(\theta)) }
\ge b n^{\frac13} \theta^{-\frac23}\bigr\}\le C b^{-3/2}+\theta e^{-b \, n^{\frac13} \theta^{-\frac23} } 
 \ee 
and
\be\label{free0.2}
 C^{-1} (n^{\frac13} \theta^{-\frac23}-\log n)\;\leq \; \E\abs{\log Z_{(1, n),(0,t)} 
-(\theta t-n \Psi_0(\theta)) }  \;\leq \; C n^{\frac13} \theta^{-\frac23}+ \theta e^{-b_0 \, n^{\frac13} \theta^{-\frac23} }
\ee
 \end{theorem}
\begin{remark}
The basic strategy of our proofs is the following. If we consider $Z_{n,t}^{\theta}$ with $\theta$ defined according to the theorem then $\E_{n,t}^\theta \sigma_0=0$. Since we expect $\sigma_0$ to be fairly close to its mean, this would suggest that $Z_{n,t}^\theta$ is fairly close to $Z_{(1,n),(0,t)}$. The main components of the proofs will rely on  comparisons between the partition functions of the two models and on the results proved about the stationary model.


Note that the centering $\theta t-n \Psi_0(\theta)$ inside probability \eqref{free0.1} is right choice, as it can be seen from the exact expression for the free energy \eqref{free-MOC} and Brownian scaling.
\end{remark}

 
\begin{proof}[Proof of the lower  bound in \eqref{free0.2}]

Note that $(n, t, \theta)$ satisfy \eqref{nt-hyp} with $\kappa=0$. 
Let  
\be f_n=\theta t-n \Psi_0(\theta)= \E ( \log Z^{\theta}_{ n,t}).
\label{f_n}\ee
Define the event $A  = \{ \log Z^{\theta}_{ n,t} \geq f_n + \delta_1
n^{\frac13} \theta^{-\frac23}  \}$.   
By Proposition \ref{prop-lower-bd} there exist $\delta_1, \delta_2>0$ such that 
\be\label{lb-again}
\P \left( A \right) \geq \delta_2. 
\ee
From (\ref{ZB2}) we get the simple bound 
\be\label{bound}
Z_{ n,t}^\para\ge e^{r_1(0)} Z_{(1, n),(0,t)}.
\ee
Utilizing this, 
\begin{eqnarray*}
f_n &=& \E [(\log Z^{\theta}_{n,t}) {\bf 1}_A] + \E [(\log Z^{\theta}_{ n,t}) {\bf 1}_{A^c}]\\
&\geq& \P(A) (f_n + \delta_1 n^{\frac13} \theta^{-\frac23}  )+\E  [  (\log
Z^{\theta}_{n,t}
-f_n ){\bf 1}_{A^c} ] + f_n \P(A^c)\\
&\geq& f_n + \delta_1 \P(A)  n^{\frac13} \theta^{-\frac23} +\E \left[ \left(\log Z_{(1,n),(0,t)}
-f_n\right){\bf 1}_{A^c}\right] + \E [{r_1(0)}{\bf 1}_{A^c}]. 
\end{eqnarray*}
Rearranging and using (\ref{lb-again}),
\begin{eqnarray*}
 \delta_1 \delta_2n^{\frac13} \theta^{-\frac23} &\leq& \delta_1 \P(A)n^{\frac13} \theta^{-\frac23}  \leq \E \left[ \left(f_n-\log Z_{(1,n),(0,t)}\right){\bf 1}_{A^c}\right] - \E {r_1(0)}{\bf
1}_{A^c}\\
&\leq& \E  \left|f_n-\log Z_{(1,n),(0,t)}\right|+ \E |{r_1(0)}|\\
&\leq& \E  \left|f_n-\log Z_{(1,n),(0,t)}\right| + C(\theta_0) \log n.
\end{eqnarray*}
We used $
\E  |{r_1(0)}|\le \sqrt{\E r_1(0)^2}=\sqrt{\Psi_1(\theta)+\Psi_0(\theta)^2} 
$
from \eqref{rk}. 
 This proves the lower bound in
(\ref{free0.2}). 
\end{proof}

\begin{proof}[Proof of the upper bounds in  \eqref{free0.1} and \eqref{free0.2}]
Inequality \eqref{bound}  gives 
\bea 
\P\left( \log Z_ { n,t}^{\para} - \log Z_{(1, n),(0,t)}   \le -b \,
n^{\frac13} \theta^{-\frac23} 
 \right)
\le \P\bigl(e^{-r_1(0)}\ge e^{b \, n^{\frac13} \theta^{-\frac23} }  \bigr)\le \theta e^{-b \, n^{\frac13} \theta^{-\frac23} }
\nn
\eea 
where the last step comes
from  Markov's inequality and   $e^{-r_1(0)}$ $\sim$ Gamma$(\theta)$.  Estimate \eqref{aa-4} gives the Chebyshev bound 
\[  \P\left( |\log Z_ { n,t}^{\para} - f_n  | \ge b \,
n^{\frac13} \theta^{-\frac23}  \right)\le 
b^{-2} n^{-\frac23} \theta^{\frac43} \Vvv(\log Z_ { n,t}^{\para}) 
\le Cb^{-2} \le C b^{-3/2}\]
which gives
\begin{align}\label{LB10}
\P\left(\log Z_{(1,n),(0,t)}-f_n\ge b \, n^{\frac13} \theta^{-\frac23}\right)\le 
C b^{-3/2}+\theta e^{-b \, n^{\frac13} \theta^{-\frac23} }.
\end{align}
The bound  on the other tail will be proved in two steps. In Lemma \ref{LBLemma3.5} below we will show that there is a $c_0>0$ depending on $\theta_0$ so that if $b>c_0 n^{2/3} \theta^{2/3}$ then 
\begin{align}\label{LB11}
\P\left(\log Z_{(1,n),(0,t)}-(\theta t-n \Psi_0(\theta))\le b \, n^{\frac13} \theta^{-\frac23}\right)\le
C e^{-C^{-1} b n^{1/3} \theta^{1/3}}.
\end{align}
In Lemma \ref{LBlemma3} below  we will show that if  $b\le c_0 n^{2/3} \theta^{2/3}$ then there are constants $C$ and $n_0$ depending on $c_0$ and $\theta_0$ so that 
\be\begin{aligned}
&\P\left(\frac{Z_{ n,t}^\para}{Z_{(1, n),(0,t)}}\ge e^{b\, n^{\frac13} \theta^{-\frac23}}\right)\le  Cb^{-3/2} 
\end{aligned} \label{free3}\ee
for all $n>n_0$.  Using the Chebyshev bound again with (\ref{free3}) and then  combining it with (\ref{LB11}) we get 
\begin{align}\label{LB12}
\P\left(\log Z_{(1,n),(0,t)}-f_n\le -b \, n^{\frac13} \theta^{-\frac23}\right)\le 
C b^{-3/2}+C e^{-C^{-1} b n^{1/3} \theta^{1/3}}.
\end{align}
The estimates  (\ref{LB10}) and (\ref{LB12}) together  establish (\ref{free0.1}).

Integrating out $b$ in (\ref{free0.1})  gives 
\[
\E  |\log Z_ { n,t}^{\para} - \log Z_{(1, n),(0,t)}  | \le C 
n^{\frac13} \theta^{-\frac23}+ C \theta e^{-b_0 \, n^{\frac13} \theta^{-\frac23} }. 
\]
Combining the above with 
 \begin{align}
\E  |\log Z_ { n,t}^{\para} -f_n | \le \sqrt{\Vvv \log Z_ { n,t}^{\para} }\le C n^{\frac13} \theta^{-\frac23}
\end{align}
verifies the upper bound of  \eqref{free0.2}. 
\end{proof}

Except for the technical estimates postponed to Section \ref{tech}, this completes  the proof of 
Theorem \ref{thm:freeZa}. 
\begin{proof}[Proof of Theorem \ref{thm:freeZ}] Recall that $\beta=\beta_0n^{-\alpha}$. 
Introduce $\tilde n=\tau n, t=\tau \beta_0^2 n^{1-2\alpha}$ and $\theta=\Psi_1^{-1}(t/\tilde n)=\Psi_1^{-1}(\beta_0^2 n^{-2\alpha})$. From Brownian scaling (\ref{bscaling}) and the explicit free energy density $\free(\beta)$ in \eqref{free-MOC}, 
\begin{align*}
&\log Z_{\tau n, \tau n}(\beta_0 n^{-\alpha})-\tau n \free(\beta)\\
&\qquad \eqd-2(\tau n-1)\log(\beta_0 n^{-\alpha})+\log Z_{(1,\tau n), (0,\beta_0^2 n^{1-2\alpha})}-\tau n \free(\beta_0 n^{-\alpha}) \\
&\qquad =\log Z_{(1,\tilde n), (0,t)}-2(\tilde n-1)\log \beta-\tilde n (\theta \beta_0^2 n^{-2\alpha}-\Psi_0(\theta)-2\log \beta)\\
&\qquad 
=\log Z_{(1,\tilde n), (0,t)}-(\theta t-\tilde n \Psi_0(\theta))+2\log \beta. 
\end{align*}
The bounds claimed in Theorem \ref{thm:freeZ} follow by quoting Theorem \ref{thm:freeZa} for $( \tilde n, t)$. Note that because $\alpha\in [0,1/4)$, the terms $\log \beta$ and $\log n$ are lower order than $n^{\frac13(1-4\alpha)}$,
and $e^{-b n^{1/3} \theta^{-2/3}}$ is lower order than $b^{-3/2}$. 
\end{proof}
We now turn to the path fluctuations for the model without boundaries. 

\begin{theorem}\label{thm:freepath}  Fix $0<\theta_0<\infty$ and 
$0<\e_0, \e_1<1/2$.   
Then there exist positive constants $b_0, C, C_1$  that 
depend only on $\theta_0, \e_0$ such that the following holds
for $n\in\bN$ and $t>0$. 
If  $\theta=\Psi_1^{-1}(t/n)$ satisfies 
\be
\theta_0\le \theta \le \theta_0^{-1} n^{1/2-\e_0} 
\label{th-e_0}\ee
then  for $n\ge n_0$, $b\ge b_0$, and  $\e_1\le  \gamma\le 1-\e_1$ we have 
\be 
P^{\beta=1}_{(1,n),(0,t)}\left(|\sigma_{\lfloor  n \gamma\rfloor}-\gamma t|>b
n^{\frac23} \theta^{-\frac43}
\right)\le C b^{-3}.
\ee 
\end{theorem}

\begin{proof}
Since $0\le \sigma_k\le t=n \Psi_1(\theta)\le c n \theta^{-1}$ we may assume that
\begin{align}\label{bUP}
b\le c n^{1/3} \theta^{1/3}
\end{align}
with a constant $c$ depending only on $\theta_0$. 

Let $\ell =\lfloor  n \gamma\rfloor,
t'=\gamma t$ and $u=b \, n^{\frac23} \theta^{-\frac43}$. 
By the definitions and  
(\ref{bound})  
\begin{align*}
&Q_{(1, n),(0,t)}\left(| \sigma_{\ell }-t'|>u  \right)=
\frac{1}{Z_{(1, n)(0,t)}}   \int\limits_{\abs{s-t'}>u}  {Z_{(1,\ell)(0,s)}
Z_{(\ell+1, n)(s,t)}}
\,ds\\
 &\quad \le \frac{e^{-r_1(0)}}{Z_{(1, n)(0,t)}}   \int\limits_{\abs{s-t'}>u}  Z^\para_{\ell
,s}\,
Z_{(\ell +1, n)(s,t)} \,ds  
\; = \; 
 \frac{e^{-r_1(0)}Z_{ n,t}^\para}{Z_{(1, n)(0,t)}} \,Q^\theta_{ n,t}\left(| \sigma_{\ell
}-t'|>u \right) .
\end{align*}
Let  $h\in(b^{-3},1)$ (note that if $b_0$ is large enough then the interval is non-empty) and set $r=\delta b^2/(3(1-\gamma))$ with $\delta$ from Lemma
\ref{main-upper-bound}. 
\bea \nn 
\P\left(Q_{(1,n),(0,t)}\left(| \sigma_{\ell }-t'|>u  \right)>h  \right)
&\le& \P( e^{-r_1(0)}\ge \theta b^{3})
+\P\left[\frac{Z_{n,t}^\para}{Z_{(1,n)(0,t)}} \ge e^{r n^{\frac13} \theta^{-\frac23}}\right]\\
\nn&&
+\P\left[ Q^\theta_{n,t}\left(| \sigma_{\ell }-t'|>u  \right) >e^{-r n^{\frac13} \theta^{-\frac23}} h \theta^{-1} b^{-3}
\right]
\eea
On the right hand side, the first term is bounded by $C b^{-3}$ by Markov's inequality, since $e^{-r_1(0)}$ $\sim$  Gamma$(\theta)$. The second term is bounded by $C r^{-3/2}\le C b^{-3}$ by (\ref{free3})
above. To see that we can actually apply Lemma \ref{LBlemma3} note that by (\ref{bUP}) we have 
\[
r=\frac{\delta b^2}{3(1-\gamma)}\le C n^{2/3} \theta^{2/3}
\]
with a constant depending on $\theta_0$ and $\e_1$ which was the condition needed for the lemma.

 Finally, the shift invariance \eqref{altQ3}  
 and Lemma \ref{main-upper-bound} give, for large
enough $n$ and $b$ and uniformly for $h\in(b^{-3},1)$,  
\begin{align*}
&\P\left[   Q^\theta_{n,t}\left(| \sigma_{\ell }-t'|>u  \right) >e^{-r n^{\frac13} \theta^{-\frac23}} h \theta^{-1} b^{-3}
\right]\le \P\left[  
Q^\theta_{ n-\ell ,t-t'}(|\sigma_0|>u) >e^{-r n^{\frac13} \theta^{-\frac23}} \theta^{-1}b^{-6}
\right]\\
&\qquad \le 
\P\left[Q^\theta_{ n-\ell ,t-t'}(|\sigma_0|>u)>e^{-\delta  \theta^2 u^2/( n-\ell) }\right] 
\le   C b^{-3}.
\end{align*}
It is above that we need $\theta\le \theta_0^{-1}n^{1/2-\e_0}$ for $\e_0>0$,
for otherwise  the right-hand side $e^{-r n^{\frac13} \theta^{-\frac23}} \theta^{-1}b^{-6}$ cannot be bounded below by $e^{-\delta  \theta^2 u^2/( n-\ell) }$. 
Collecting the estimates gives 
\[
\P\left[Q_{(1,n),(0,t)}\left(| \sigma_{\ell }-t'|>u  \right)>h  \right]\le C  b^{-3}
\]
and from this 
\begin{align*}
&P_{(1, n),(0,t)}\left(|\sigma_{\lfloor  n \gamma\rfloor}-\gamma t|>b
n^{\frac23} \theta^{-\frac43}
\right)\\ &\qquad\qquad \le b^{-3}
 +\int_{b^{-3}}^1 \P\left[Q_{(1, n),(0,t)}\left(| \sigma_{\ell }-t'|>u  \right)>h 
\right] dh\;\le\; C b^{-3}.
\end{align*}
This completes the proof.
\end{proof}
\begin{proof}[Proof of Theorem \ref{thm:path-ptp}]
We again introduce $\tilde n=\tau n, t=\tau \beta_0^2 n^{1-2\alpha}$ and $\theta=\Psi_1^{-1}(t/\tilde n)=\Psi_1^{-1}(\beta_0^2 n^{-2\alpha})$.  Assumption 
\eqref{th-e_0} is satisfied because $\alpha<1/4$. 
Using (\ref{bscaling}) and Theorem \ref{thm:freepath} the theorem follows. 
\end{proof}
\subsection{The tail estimates}\label{tech}
In this section we prove the missing components of the proofs of Theorem \ref{thm:statpath} and Theorem \ref{thm:freeZa}.
We begin with some definitions.

Augment the family $Z_{(j,k),(s,t)}=Z_{(j,k),(s,t)}(1)$ defined
for $j\ge 1$  in \eqref{zetadef1}
by introducing, for $k\in \bN$ and $t\in\bR_+$, 
\be  Z_{(0,0),(0,t)}=e^{-B(t)}, \qquad  Z_{(0,k),(0,t)} = \hskip-20pt
\int\limits_{0<s_{0}<\dotsm<s_{k-1}<t} \hskip-20pt  \exp\bigl[ -B(s_0) +B_1(s_0,s_1) 
+\dotsm + B_k(s_{k-1},t)\bigr] \,ds_{0,k-1}.
\label{Zdef1.02}\ee
It is also convenient  to set, for $A\subseteq\bR$,  
\be Z_{0,t}^\para(\sigma_0\in A)= \ind_{A\cap\bR_+}(t) \exp[-B(t)+\para t].  \label{0conv}\ee 
The following bounds are proved in Lemma 3.8  of \cite{SV} .
\begin{lemma}\label{lemma-comp}\cite{SV}  Let $\para>0$.  For $0< s<t$ and $n\in\bZ_+$
\be
\frac{Z_{n+1,t}^\para(\sigma_0>0)}{Z_{n,t}^\para(\sigma_0>0)}
\le \frac{Z_{(0,n+1),(0,t)}}{Z_{(0, n),(0,t)}}\le
\frac{Z_{n+1,t}^\para(\sigma_0<0)}{Z_{n,t}^\para(\sigma_0<0)}
\label{comp1}\ee
and 
\be
\frac{Z_{n,t}^\para(\sigma_0>0)}{Z_{n,s}^\para(\sigma_0>0)}
\ge \frac{Z_{(0,n),(0,t)}}{Z_{(0, n),(0,s)}}\ge \frac{Z_{n,t}^\para(\sigma_0<0)}{Z_{n,s}^\para(\sigma_0<0)}. 
\label{comp2}\ee
The second inequality of \eqref{comp2} makes sense only for $n\ge 1$. 
\end{lemma} 
For  $A\subset \bR$ note the identity  
\be \nn 
\frac{Z^\para_{n,t}( \sigma_0\in A)}{Z_{1,n}(0,t)}=\int_A \exp(-B(s)+\para s)
\frac{Z_{1,n}(s,t)}{Z_{1,n}(0,t)} ds.
\ee 
  We will also define a reversed system: construct a new environment $\tilde\om$ 
with 
\be\nn 
\tilde B(s)=-(B_n(t)-B_n(t-s)), \quad \tilde B_i(s)=B_{n-i}(t)-B_{n-i}(t-s), \quad 1\le i \le n-1. 
\ee
Quantities  that use environment $\tilde\om$ are marked with a tilde. 
From the definitions one checks that 
\be
Z_{1,n}(s,t)=\tilde Z_{0,n-1}(0,t-s) \quad\text{ for any $t>0$ and  $s\in(-\infty,t)$.} 
\label{ZZtil}\ee

\begin{lemma} \label{LBlemma3}  Let $\theta=\Psi_1^{-1}(t/n)$ and assume that $\theta_0\le \theta\le \theta_0^{-1} \sqrt{n}$ with a fixed $\theta_0>0$. Fix a $c_0>0$. 
Then there exist finite, positive  
constants $C, n_0$ depending on $\para_0, c_0$ such that if $ n>n_0$ and $b\le c_0 n^{2/3} \theta^{2/3}$, then 
\begin{eqnarray}\label{LBlemma3-1}
\P\left(\frac{Z_{ n,t}^\para}{Z_{(1, n),(0,t)}}\ge e^{b\, n^{\frac13} \theta^{-\frac23}}\right)\le C b^{-3/2}.
\end{eqnarray}
\end{lemma} 
\begin{proof}
Note that once we prove (\ref{LBlemma3-1}) for $b>b_0$ with a constant $b_0$ depending on $\theta_0, c_0$ then we can get it for all $b$ by adjusting the constant $C$. Thus we may assume that $b$ is big enough compared to $\theta_0$ and $c_0$.

Let $u=\sqrt{b} n^{2/3} \theta^{-4/3}$ and $\nu=\eps \sqrt{b} n^{-1/3} \theta^{2/3}$ where $\eps>0$ will be specified later.  Then
\be\begin{aligned}
&\P\left(\frac{Z_{ n,t}^\para}{Z_{(1, n),(0,t)}}\ge e^{b\, n^{\frac13} \theta^{-\frac23}}\right)=\P\left(\frac{Z_{ n,t}^\para(|\sigma_0|\le
u)}{Z_{(1, n),(0,t)}\,
Q_{ n,t}^\para(|\sigma_0|\le u)}\ge  e^{b\, n^{\frac13} \theta^{-\frac23}}  \right)\\
& \qquad \qquad \le \P\left( \frac{Z_{ n,t}^\para(|\sigma_0|\le u)}{Z_{(1, n),(0,t)} }\ge
\frac12 e^{b\, n^{\frac13} \theta^{-\frac23}}  \right)+
\P\left(Q_{ n,t}^\para(|\sigma_0|\le u)\le 1/2\right)
\end{aligned} \ee
The second probability can be bounded as
\begin{align}
\P\left(Q_{ n,t}^\para(|\sigma_0|\le u)\le 1/2\right)=\P\left(Q_{ n,t}^\para(|\sigma_0|> u)\ge 1/2\right)\le C b^{-3/2}
\end{align}
by (\ref{aa-1}) of Lemma \ref{main-upper-bound}
The first probability can be bounded by
\[
 \P\left( \frac{Z_{ n,t}^\para(0\le \sigma_0\le  u)}{Z_{(1, n),(0,t)} }\ge
\frac14 e^{b\, n^{\frac13} \theta^{-\frac23}}  \right)+ \P\left( \frac{Z_{ n,t}^\para(-u\le \sigma_0<0)}{Z_{(1, n),(0,t)} }\ge
\frac14 e^{b\, n^{\frac13} \theta^{-\frac23}}  \right).
\]
We will bound the first term, the second will follow similarly. 

Introduce the new parameter $\lambda=\theta-\nu$. Note that by choosing $\eps^2\le (4 c_0)^{-1}$ we can assume $\lambda>\theta/2$. 
  
Begin with \eqref{ZZtil}   and then apply   comparison  \eqref{comp2}: 
\begin{align*}
\frac{Z_{(1,n),(s,t)}}{Z_{(1,n),(0,t)}}& =\frac{\tilde Z_{(0,n-1),(0,t-s)}}{\tilde
Z_{(0,n-1),(0,t)}} \le
\frac{\tZ_{n-1,t-s}^{\lambda}(\sigma_0<0)}{\tZ_{n-1,t}^{\lambda}(\sigma_0<0)}=
\frac{\tZ_{n-1,t-s}^{\lambda}}{\tZ_{n-1,t}^{\lambda}}\cdot \frac{\tilde
Q^{\lambda}_{n-1,t-s}(\sigma_0<0)}{\tilde Q^{\lambda}_{n-1,t}(\sigma_0<0)}\\
&= \exp\bigl(\tilde Y_{n-1}(t-s,t)-\lambda s\bigr)\cdot \frac{\tilde
Q^{\lambda}_{n-1,t-s}(\sigma_0<0)}{\tilde Q^{\lambda}_{n-1,t}(\sigma_0<0)}\\
&\le  \exp\bigl(\tilde Y_{n-1}(t-s,t)-\lambda s\bigr)\cdot \frac{1}{\tilde
Q^{\lambda}_{n-1,t}(\sigma_0<0)}. 
\end{align*}
where we used (\ref{Y-increment}) for the reversed system.
Specializing the above to our context and substituting it in the probability that is to
be bounded:  
\begin{align} \nn
&\P\left(\frac{Z^\para_{n,t}( 0< \sigma_0\le u)}{Z_{(1,n),(0,t)}}\ge \frac14
e^{n^{\frac13} \theta^{-\frac23}\bbb }\right) =\P\left(\int_0^u \exp(-B(s)+\para s)
\frac{Z_{(1,n),(s,t)}}{Z_{(1,n),(0,t)}}\, ds
\ge \frac14 e^{n^{\frac13} \theta^{-\frac23} \bbb}  \right)\\[2pt]
&\qquad\qquad  \nn \le
\P\left(\int_0^u \frac{ \exp(-B(s)+\tilde Y_{n-1}(t-s,t)+(\para-\lambda) s)  }{\tilde
Q^{\lambda}_{n-1,t}(\sigma_0<0)}\,ds\ge\frac14 e^{n^{\frac13} \theta^{-\frac23}
\bbb}\right)\\
&\qquad\qquad 
\le\P\left(\tilde Q^{\lambda}_{n-1,t}(\sigma_0<0)\le 1/2\right)\label{eq:lower3.9}
 \\
&\qquad\qquad 
\qquad+ \ \P\left(\int_0^u \exp(-B(s)+\tilde Y_{n-1}(t-s,t)+\nu s) \, ds\ge
\frac1{8} 
e^{n^{\frac13} \theta^{-\frac23} \bbb}\right).\label{eq:lower4}
\end{align}
To treat probability (\ref{eq:lower3.9}) set 
\begin{align*}
\bar u=  (n-1)\trigamf(\lambda)-n \trigamf(\para)&\ge -\trigamf(\para) +(n-1)\Psi_2(\para)(\lambda-\para)  \\
&\ge   -C \theta^{-1} + \tfrac12 C\eps \sqrt{b} n^{\frac23}\theta^{-\frac43}   
\ge C' \sqrt{b}  n^{\frac23} \theta^{-\frac43},
\end{align*}
where  we used our assumptions on $\theta$, the bounds \eqref{Psi-3}, 
and took $b$ large enough in relation to $\theta_0$. 
Use  
invariance   (\ref{altQ2}) of $Q$  and   upper bound (\ref{aa-1}):
\be\begin{aligned}
\P(\tilde Q^{\lambda}_{n-1,t}(\sigma_0<0)<1/2)&=\P(\tilde
Q^{\lambda}_{n-1,t}(\sigma_0>0)\ge1/2)\\
&=
\P(\tilde Q^{\lambda}_{n-1,t+\bar u}(\sigma_0>\bar u)\ge1/2) \le C(\theta_0) b^{-3/2}.
 \end{aligned}\label{auxlemma8}\ee
To justify our use of the upper bound, note that  
\[   {(n-1)\Psi_1(\lambda)-t-\bar u\,} = {n\Psi_1(\theta)-t\,} =0 \]
 so the upper bound \eqref{aa-1} is valid for $\bar u$.

For probability (\ref{eq:lower4}),  observe first that 
  $s\mapsto \tilde Y_{n-1}(t-s,t)$  is 
  a standard Brownian motion which is independent of $B$ by construction. By introducing 
$
B^{\dagger}(s) = \tfrac1{\sqrt{2}} \bigl(-B(s)+\tilde
Y_{n-1}(t-s,t)\bigr)$ we need to bound
\[
 \P(\int_0^u \exp(\sqrt{2} B^{\dagger}((s)+\nu s)ds \ge \frac18 e^{\eps^{-1} \nu u} )\le  \P(\int_0^u \exp(\sqrt{2} B^{\dagger}((s)+\nu s)ds \ge  e^{3 \nu u} ).
\]
where the upper bound follows by choosing $\eps$ small enough (for fixed $b_0, n_0, \theta_0$). 
We will show that 
\begin{align} \P(\int_0^u \exp(\sqrt{2} B^{\dagger}((s)+\nu s)ds \ge e^{3\nu u})\le C e^{- \tfrac14 \nu^2 u} \label{intbnd}
\end{align}
 if  $\nu>0$, $u>0$.  Note that
 \begin{align*}
  &\P(\int_0^u \exp(\sqrt{2} B^{\dagger}((s)+\nu s)ds \ge e^{3\nu u})\le  \P(\int_{-\infty}^u \exp(\sqrt{2} B^{\dagger}((s)+\nu s)ds \ge e^{3\nu u})\\
  &\qquad \qquad \qquad 
  \le \P(\exp(\sqrt{2}B^{\dagger}((u)+\nu u)\ge e^{2\nu u})\\
  &\qquad  + \P(\int_{-\infty}^u \exp(\sqrt{2} (B^{\dagger}((s)-B^{\dagger}((u))+\nu (s-u))ds \ge e^{\nu u}).
 \end{align*}
 The first probability is $\P(\sqrt{2}B^{\dagger}((u)\ge \nu u)\le C \exp(-\nu^2 u /4)$. For the second probability we note that by Dufresne's identity \cite{dufr-osaka01} the integral has the same distribution as the reciprocal of a $\textup{Gamma}(\nu)$ random variable. Thus the second term is
 \begin{align*}
 \P(\textup{Gamma}(\nu)\le e^{-\nu u})\le \frac{1}{\nu \Gamma(\nu)} e^{-\nu^2 u} \le C e^{-\nu^2 u}
 \end{align*} 
 which proves the estimate (\ref{intbnd}).

 Collecting everything we get that 
 \[
 P\left(\frac{Z^\para_{n,t}( 0< \sigma_0\le u)}{Z_{(1,n),(0,t)}}\ge \frac14
e^{n^{\frac13} \theta^{-\frac23}\bbb }\right) \le C b^{-3/2}+C e^{-\tfrac14 \eps^2 b^{3/2}}\le C' b^{-3/2}.
 \]

The case of $-u<\sigma<0$ goes similarly, with small alterations.
Now $\lambda=\para+\nu\le 3\theta/2$.  Utilizing 
\eqref{ZZtil} and comparison \eqref{comp2}
the ratio is  developed as follows:
\begin{align*}
\frac{Z^\para_{n,t}( -u\le  \sigma_0<0)}{Z_{(1,n),(0,t)}} 
&=
 \int_{-u}^0 \exp(-B(s)+\para s) \frac{Z_{(1,n),(s,t)}}{Z_{(1,n),(0,t)}}\, ds\\
 &  \le  \int_{-u}^0 \frac{ \exp(-B(s)-\tilde Y_{n-1}(t,t-s)-(\para-\lambda) s)  }{\tilde
Q^{\lambda}_{n-1,t}(\sigma_0>0)}\,ds.
\end{align*}
The rest follows along the same lines as above. This proves (\ref{LBlemma3-1}).  
\end{proof}

\begin{lemma}\label{LBLemma3.5}
Fix $\theta_0>0$, suppose that $\theta_0<\theta<\theta^{-1} \sqrt{n}$ and let $t=n \Psi_1(\theta)$. 
Then there exist constants $c_0, C$ depending on $\theta_0$ so that 
\begin{align}
P\left(\log Z_{(1,n),(0,t)}- (\theta t-n \Psi_0(\theta))<-x\right)\le c_1 e^{-c_1^{-1} x  \theta }, \qquad \textup{for $x\ge c_0 n$. } 
\end{align}
The same bound holds for the upper tail. 
\end{lemma}
\begin{proof}

We first note that
\begin{align*}
Z_{(1,n)(0,t)}&= \int_{0<s_1<\dots<s_{n-1}<t} \exp\left\{ B_1(0,s_1)+\dots+B(s_{n-1},1)\right\} ds_{1,n-1}\\
&>\frac{t^{n-1}}{(n-1)!}  \exp\left( \min_{0<s_1<\dots<s_{n-1}<t} ( B_1(0,s_1)+\dots+B(s_{n-1},1))    \right)\\
&\eqd \frac{t^{n-1}}{(n-1)!} \exp(-\sqrt{n t} \, \lambda^n_{\textup{max}})
\end{align*}
where $\lambda^n_{\textup{max}}$ is the largest eigenvalue of an $n\times n$ $GUE$ random matrix where the non-diagonal entries have variance $1/n$. (This is the normalization where the support of the spectrum converges to $[-2,2]$.) The fact that 
\begin{align*}
 \min_{0<s_1<\dots<s_{n-1}<1} ( B_1(0,s_1)+\dots+B(s_{n-1},1))& =- \max_{0<s_1<\dots<s_{n-1}<1} ( B_1(0,s_1)+\dots+B(s_{n-1},1)) \\&
 \eqd -\sqrt{n } \, \lambda^n_{\textup{max}}
\end{align*}
was proved independently in \cite{GTW} and \cite{Bar}.

It is known that the random variable $\evmax$ converges to 2, and the following uniform tail bound holds in $n$ for $K>K_0>2$ (see e.g. \cite{Led} and the references within): 
\begin{align}
P(\evmax>K)\le C e^{-C^{-1} n K^{2}}\label{GUEtail} 
\end{align}
with a constant $C$ depending only on $K_0$. We will use this bound with $K_0=3$.

We have
\be \begin{aligned} 
&P\left(\log Z_{(1,n),(0,t)}- (\theta t-n \Psi_0(\theta))<-x\right) \\
&\qquad \le P\left(e^{- (\theta t-n \Psi_0(\theta))} \frac{t^{n-1}}{(n-1)!} \exp(-\sqrt{n t} \, \lambda^n_{\textup{max}})<e^{-x}\right)\\
& \qquad = P\left(-\theta t+n \Psi_0(\theta)+
x+(n-1) \log t-\log(n-1)!<\sqrt{n t} \evmax
\right).  \end{aligned}
\label{yyy}\ee
Using Stirling's formula, the bounds $\theta_0\le \theta\le \theta_0^{-1} n^{1/2}$ and the bounds (\ref{Psi-3}), (\ref{didam1}) on $\Psi_0, \Psi_1$ we get that 
\begin{align*}
\left|-\theta t+n \Psi_0(\theta)+(n-1) \log t-\log(n-1)!\right| \le C n
\end{align*}
where  $C$ depends on $\theta_0$. This gives
\begin{align*}
P\left(\log Z_{(1,n),(0,t)}- (\theta t-n \Psi_0(\theta))<-x\right)&\le P\left(-C n+x<\sqrt{n t} \evmax
\right)
\end{align*}
Choosing $c_0>2C$ we get
\begin{align*}
P\left(-C n+x<\sqrt{n t} \evmax
\right)
&\le P\left(\tfrac12x<\sqrt{n t} \evmax
\right)
\le P(\tfrac12 x n^{-1} \Psi_1(\theta)^{-1/2}<\evmax).
\end{align*}
Now choose $c_0$ big enough so that  $\tfrac12 c_0 \Psi_1(\theta_0)^{-1/2}>3$ (this is possible since $\theta_0\le \theta$) so that we can use (\ref{GUEtail}):
\begin{align*}
P\left(\log Z_{(1,n),(0,t)}- (\theta t-n \Psi_0(\theta))<-x\right)&\le P(\frac12 x n^{-1} \Psi_1(\theta)^{-1/2}<\evmax)\\
&\le C e^{-C^{-1} n (\frac12 x n^{-1} \Psi_1(\theta)^{-1/2})^{2}}
\\
&\le c_1 e^{-c_1^{-1} \theta x}
\end{align*}
where in the last step we used the bounds on $\Psi_1(\theta)$ and $x\ge c_0 n$. 
\end{proof}

\section{Proofs for the KPZ equation}\label{pf-KPZ}
\begin{proof}[Proof of Theorem \ref{kpzclose}]
We first start with the case $\vf=0$. From (\ref{var-ident}) one can verify that
\begin{eqnarray*}
 \E \log \she_N(\tau) \leq A \tau,
\end{eqnarray*}
for some constant $A>0$. This, together with the upper bound in Theorem \ref{thm-unscaled} yields
\begin{eqnarray*}
 \Vvv \log \she_N(\tau) \leq F(\tau),
\end{eqnarray*}
for some finite $F(\tau)$. This implies that $\{\log \she_N(\tau):\, N\geq 1\}$ is uniformly
integrable, and hence $\E \log \she_N(\tau) \to \E \log \she(\tau,0)$. Note that this could have
been obtained by exact computations for the stochastic heat equation as well. Together with Theorem \ref{thm:monster}, this gives the
convergence in law 
$$\log \she_N(\tau)-\E\log \she_N(\tau) \Rightarrow \log \she(\tau)-\E \log \she(\tau).$$
Fatou's lemma and Theorem \ref{thm-unscaled} then give
\begin{eqnarray*}
\Vvv \log \she(\tau) &\leq& \liminf_N \E \left[\left(\log \she_N(\tau)-\E\log
\she_N(\tau)\right)^2\right]\\
&\leq& C \tau^{\frac23},
\end{eqnarray*}
for some $C>0$. As for the lower bound, 
\begin{eqnarray*}
\P\left\{ \log \she(\tau) -  \E \log \she(\tau) \geq c \tau^{1/3}\right\}
&=& \lim_n \P\left\{ \log \she_n(\tau) - \E \log \she_n(\tau) \geq c \tau^{1/3}
\right\}\geq \delta,
\end{eqnarray*}
for some constants $c,\, \delta>0$, and $\tau$ large enough, thanks to the lower bound in
Proposition \ref{prop-lower-bd}.
We have proved that there exist some constant $C>0$ such that
\begin{eqnarray*}
\frac{1}{C} \tau^{\frac23} \leq \Vvv \log \she(\tau) &\leq&C \tau^{\frac23},
\end{eqnarray*}
for $\tau$ large enough.

We now turn to the case $|\vf|\leq K $ for some $0<K<+\infty$. From (\ref{shevf3}), we can verify
that
\begin{eqnarray*}
 e^{-K} \she^{\vf}_N(\tau) \leq \she_N(\tau)\leq e^{K} \she^{\vf}_N(\tau).
\end{eqnarray*}
This implies that
\begin{eqnarray*}
\Vvv \log \she^{\vf}_N(\tau) \leq 8K^2 + 2 \Vvv \log \she_N(\tau), 
\end{eqnarray*}
which in turns implies the uniform integrability of $\{\log \she^{\vf}_N(\tau):\, N\geq 1\}$.
Fatou's lemma and the upper bound on $\Vvv \log \she_N(\tau)$ show that
\begin{eqnarray*}
\Vvv \log \she^{\vf}(\tau) \leq C' \tau^{\frac23}, 
\end{eqnarray*}
for some $C'>0$ and $\tau>0$ large enough. The lower bound follows from Proposition
\ref{prop-lower-bd}
\begin{eqnarray*}
\P\left\{ \log \she^{\vf}(\tau) -  \E \log \she^{\vf}(\tau) \geq c \tau^{1/3}\right\}
&=& \lim_n \P\left\{ \log \she_N^{\vf}(\tau) - \E \log \she_N^{\vf}(\tau) \geq c \tau^{1/3}
\right\}\\
&\geq& \lim_n \P\left\{ \log \she_N(\tau) - \E \log \she_N(\tau) \geq c \tau^{1/3}-2K
\right\}\\\\
&\geq& \delta,
\end{eqnarray*}
for suitable $c,\, \delta>0$ and all $N$ and $\tau$ large enough. This completes  the proof of the theorem.
\end{proof}


\section{Facts about the gamma function and Gamma distribution}\label{gtb} We collect here some basic facts about the gamma function and the Gamma distribution.

%

Recall that $\Psi_0=\Gamma'/\Gamma$ and $\Psi_n=\Psi_{n-1}'$ for $n\ge 1$. The lemma below follows from straightforward computations. 
\begin{lemma}
Let $A\sim$ Gamma$(\mu,r)$. Then we have
\begin{align*}
  \E A=\frac{\mu}{r},\quad  \Vvv(A) = \frac{\mu}{r^2}, \quad
   \E(\log A) =  \Psi_0(\mu)-\log r,\quad
 \Vvv(\log A) = \Psi_1(\mu).
\end{align*}
 \end{lemma}
The polygamma functions satisfy 
\be \Psi_n(x)=(-1)^{n+1}n! \sum_{k\ge 0}(x+k)^{-n-1},  \qquad\text{for $n\ge 1$.}
\label{Psi-2}\ee
From this  it follows that for $n\ge 1$ we have 
\be \frac{(n-1)!}{x^n} \le  \abs{\Psi_n(x)} \le \frac{(n-1)!}{x^n}+\frac{n!}{x^{n+1}}
\label{Psi-3}\ee
For $\Psi_0$ we have the following asymptotics for large $x>0$:  
\be  \Psi_0(x) = \log x - \frac{1}{2x} + O\left(\frac{1}{x^2}\right).  \label{didam1}\ee

%

\end{document}